\documentclass{article}

\usepackage[a4paper,left=20mm,top=25mm,right=20mm,bottom=35mm]{geometry}

\usepackage{authblk}
\usepackage[numbers, sort&compress]{natbib}
\usepackage[margin=30pt,font=small,labelfont=bf]{caption}
\usepackage{amsmath,amsopn,amsthm,amssymb}

\usepackage{graphicx}
\usepackage{mathptmx} 
\usepackage[mathscr]{euscript}
\usepackage{mathtools}
\usepackage{wasysym}
\usepackage{float}
\usepackage{enumitem}
\usepackage{xspace}
\usepackage{marvosym}

\usepackage[
 	colorlinks=true,
	urlcolor=black,
	linkcolor=blue
]{hyperref}

\usepackage{algorithm}
\usepackage[noend]{algpseudocode}
\algrenewcommand\algorithmicrequire{\textbf{Input:}}
\algrenewcommand\algorithmicensure{\textbf{Output:}}

\newtheorem{theorem}{Theorem}[section]
\newtheorem{lemma}[theorem]{Lemma} 
\newtheorem{corollary}[theorem]{Corollary}
\newtheorem{proposition}[theorem]{Proposition}

\newtheorem{definition}[theorem]{Definition}
\newtheorem{observation}[theorem]{Observation}

\newtheorem{remark}[theorem]{Remark}

\DeclareRobustCommand*{\join}{\mathrel{\ooalign{\hss$\triangleleft$\hss\cr$\triangleright$}}}

\newcommand{\Hasse}[1][]{\mathscr{H}\ifthenelse{\equal{#1}{}}{}{(#1)}}
\newcommand{\hasse}[1][]{\mathscr{H}\ifthenelse{\equal{#1}{}}{}{(#1)}}
\DeclareMathOperator{\CC}{\mathtt{C}}

\newcommand{\PrimeCat}{\ensuremath{\vcenter{\hbox{\includegraphics[scale=0.01]{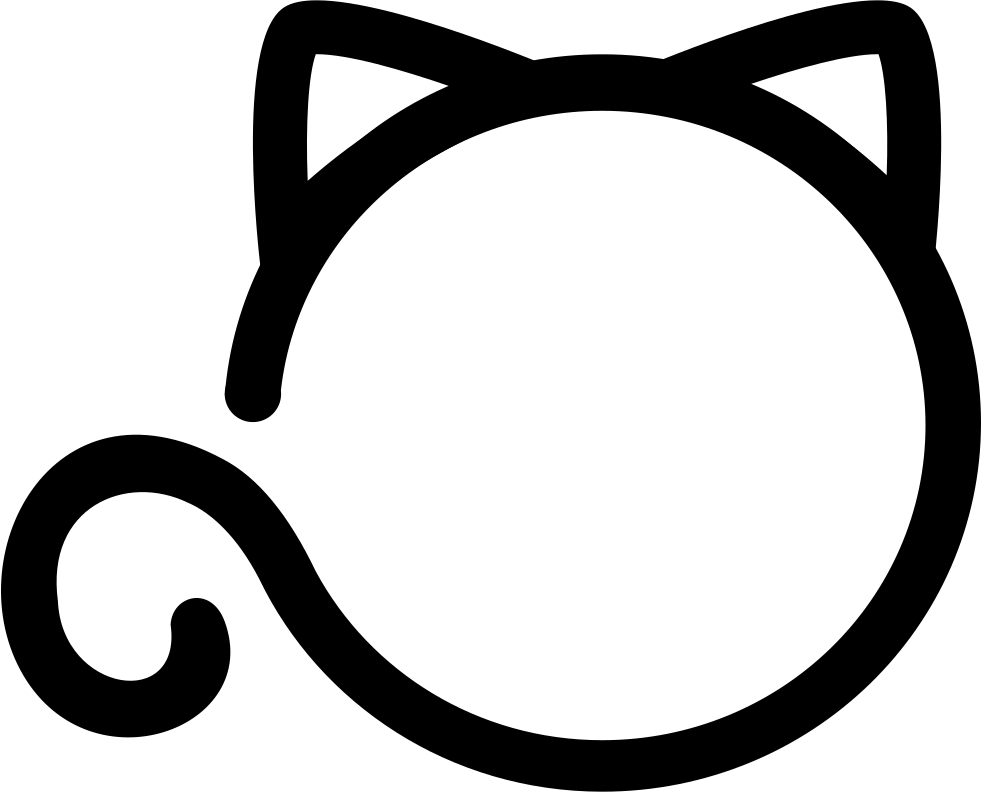}}}^{\textit{prime}}}}
\newcommand{\gatex}{\textsc{GaTEx}\xspace}
\DeclareMathOperator{\child}{child}
\newcommand{\LCA}{\ensuremath{\operatorname{LCA}}}
\newcommand{\lca}{\ensuremath{\operatorname{lca}}}
\newcommand{\parent}{\ensuremath{\operatorname{par}}}
\DeclareMathOperator{\indeg}{indeg}
\DeclareMathOperator{\outdeg}{outdeg}

\newcommand{\MD}{\ensuremath{\mathbb{M}}}
\newcommand{\MDstrong}{\ensuremath{\mathbb{M}_{\mathrm{str}}}}
\newcommand{\Mmax}{\ensuremath{\mathbb{M}_{\max}}}
\newcommand{\MDT}{\ensuremath{\mathscr{T}}}
\newcommand{\N}{\ensuremath{\mathscr{N}}}
\renewcommand{\P}{\ensuremath{\mathcal{P}}}

\DeclareMathOperator{\klca}{\textit{k}-lca}
\DeclareMathOperator{\2lca}{\textit{2}-lca}

\newcommand{\primeL}{\textsc{Prime}\xspace}
\DeclareMathOperator{\pvr}{pvr}
\newcommand{\pfam}{\mathcal{F}}

\DeclareFontFamily{OT1}{pzc}{}
\DeclareFontShape{OT1}{pzc}{m}{it}{ <-> s*[1.1] pzcmi7t }{}
\DeclareMathAlphabet{\mathpzc}{OT1}{pzc}{m}{it}
\newcommand{\SM}{\mathpzc{S}}
\newcommand{\TM}{\mathpzc{T}}

\newenvironment{owndesc}%
	{\begin{description}[leftmargin = 0.2cm, labelsep = 0.2cm]}
    {\end{description}}

\providecommand{\keywords}[1]{\textbf{\textit{Keywords: }} dissimilarities, symbolic ultrametrics, 2-structures, edge-colored graphs, strudigram, prime-vertex replacement, phylogenetic network, clustering systems}

\begin{document}

\title{Network Representation and Modular Decomposition of Combinatorial Structures: A Galled-Tree Perspective}

\author[1]{Anna Lindeberg} 
\author[2]{Guillaume E. Scholz} 
\author[1,*]{Marc Hellmuth} 

\affil[1]{Department of Mathematics, Faculty of Science,
  Stockholm University, SE-10691 Stockholm, Sweden} 

\affil[2]{Bioinformatics Group, Department of Computer Science \&
    Interdisciplinary Center for Bioinformatics, Universit{\"a}t Leipzig,
    H{\"a}rtelstra{\ss}e~16--18, D-04107 Leipzig, Germany.}

\affil[*]{corresponding author}

\date{\ }

\setcounter{Maxaffil}{0}
\renewcommand\Affilfont{\itshape\small}

\maketitle 
\abstract{In phylogenetics, reconstructing rooted trees from distances between taxa is a common
          task. B\"ocker and Dress generalized this concept by introducing symbolic dated maps
          $\delta\colon X \times X \to \Upsilon$, where distances are replaced by symbols, and
          showed that there is a one-to-one correspondence between symbolic ultrametrics and labeled
          rooted phylogenetic trees. Many combinatorial structures fall under the umbrella of symbolic
          dated maps, such as 2-dissimilarities, symmetric labeled 2-structures, or edge-colored
          complete graphs, and are here referred to as strudigrams.

			Strudigrams have a unique decomposition into non-overlapping modules, which can be
			represented by a modular decomposition tree (MDT). In the absence of prime modules,
			strudigrams are equivalent to symbolic ultrametrics, and the MDT fully captures the
			relationships $\delta(x,y)$ between pairs of vertices $x,y \in X$ through the label of
			their least common ancestor in the MDT. However, in the presence of prime vertices, this
			information is generally hidden.

			To provide this missing structural information, we aim to locally replace the prime
			vertices in the MDT to obtain networks that capture full information about the strudigrams.
			While starting with the general framework of prime-vertex replacement networks, we then
			focus on a specific type of such networks obtained by replacing prime vertices with
			so-called galls, resulting in labeled galled-trees. We introduce the concept of galled-tree
			explainable (\textsc{GaTEx}) strudigrams, provide their characterization, and demonstrate that
			recognizing these structures and reconstructing the labeled networks that explain them can
			be achieved in polynomial time.
			}


\section{Introduction}

It is a common task in phylogenetics to reconstruct rooted trees with leaf set $X$, reflecting the
evolutionary relationships between species, genes, or other taxa collected in $X$, from distances
$d\colon X \times X \to \mathbb{R}$ among them. In particular, there is a well-known
one-to-one correspondence between so-called rooted dated phylogenetic trees $(T,t)$ and
ultrametrics, i.e., metrics satisfying the strong triangle inequality $d(x, y) \leq \max\{d(x, z),
d(y, z)\}$ for all $x, y, z \in X$ \cite{Gordon:87,sem-ste-03a}. In this case, the date $t(\lca(x,
y))$ assigned to the unique least common ancestor $\lca(x, y)$ of $x$ and $y$ in $T$ coincides
with the distance $d(x, y)$ for all $x, y \in X$.

In their seminal paper \cite{BD98}, B\"ocker and Dress considered a combinatorial generalization of ultrametrics
and removed the restrictions regarding the values
of $t$ and $d$ as far as possible. Instead of considering real numbers $d(x, y)$, they considered
symbolic dating maps $\delta\colon X \times X \to \Upsilon$, i.e., symmetric maps assigning to each pair $(x, y)$ some
symbol $\delta(x, y) \in \Upsilon$. In particular, they showed the one-to-one correspondence between
so-called symbolic ultrametrics and symbolic dated, or simply labeled, phylogenetic trees $(T, t)$. 
Symbolic ultrametrics $\delta\colon X \times X \to \Upsilon$ are defined by the following two properties:
\begin{description}\smallskip
\item[(U2)] There are no four vertices $x, y, u, v \in X$ such that

				 $\delta(x, y) = \delta(y, u) = \delta(u, v) \neq \delta(y, v) = \delta(x, v) = \delta(x, u)$. \smallskip
\item[(U3)] $|\{\delta(x, y), \delta(x, z), \delta(y, z)\}| \leq 2$ 
					for all $x, y, z \in V$.\smallskip
\end{description}

While long considered of interest only from a theoretical point of view\footnote{Anecdotally,
Sebastian B\"ocker mentioned in a private communication that he never thought that symbolic ultrametrics would have
any practical relevance.} (cf.\ e.g.\ \cite{SEMPLE1999300,Dress2007,DS:2007}), such symbolic dating maps have become relevant in practice and are shown
to play a crucial role in the understanding of so-called orthologous and paralogous genes
\cite{Hellmuth:13a,HS-chapter-2024} and the reconstruction of gene and species trees \cite{Hellmuth:15a}.

Restricting these symmetric maps $\delta\colon X \times X \to \{0,1\}$ to be binary, there
is a direct translation to ``usual'' graphs $G=(X,E)$ with vertex set $X$ and adjacent vertices $x$
and $y$ precisely if $\delta(x, y)=1$. It is well known \cite{Hellmuth:13a} that a graph $G$ represents a 
symbolic ultrametric if and only if $G$ is a cograph, i.e., $G$
does not contain induced paths $P_4$ on four vertices \cite{Corneil:85,CLS:81}. 
In this case, $G$ can be ``explained'' by a $0/1$-labeled tree $(T,t)$ 
in the sense that $x,y \in X$ are adjacent if and only if $t(\lca(x, y))=1$. 
In a more general setting, symbolic ultrametrics can be characterized as edge-colored graphs that do not contain so-called
monochromatic $P_4$s (due to Condition (U2)) and rainbow triangles (due to Condition (U3)) \cite{HSW:17, Hellmuth:13a}, see Theorem \ref{thm:char-treeEx}. 
A similar result was provided by Gurvich in \cite[Thm.\ 3]{Gurvich:09} for the 
characterization of the normal forms of so-called positional games with perfect information.

Independent of the work of B\"ocker and Dress, Ehrenfeucht and Rozenberg established the theory of labeled 2-structures
\cite{Ehren1990A,Ehren1990B}. 
This theory relates to clustering and tree-representation, and is closely connected to symbolic
dating maps. As shown in \cite{HSW:17}, a symbolic dating map
$\delta$ is a symbolic ultrametric if and only if the translation of $\delta$
into symmetric labeled 2-structures is ``uniformly non-prime''.
The latter is closely related to the notion of modules or clans \cite{EHREN1994,ehrenfeucht1999theory,Engelfriet1996,HSW:17}, i.e., 
subsets $M \subseteq X$ for which $\delta(x, z)=\delta(y, z)$ is satisfied for all $x, y \in M$ and $z \in X \setminus M$. 
In particular, for each symbolic dating map $\delta$, there is a unique set system $\mathfrak{C}_{\delta}$ of non-overlapping modules 
called the \emph{modular decomposition} of $\delta$. Since no two elements in $\mathfrak{C}_{\delta}$ overlap, 
$\mathfrak{C}_{\delta}$ is a hierarchy and thus, the Hasse diagram $T_{\delta} \coloneqq \hasse(\mathfrak{C}_{\delta})$
is a phylogenetic tree \cite{Dress:97,Hellmuth2023,birkhoff1940lattice} that can be equipped with a labeling $t_{\delta}$ resulting in a labeled tree $(T_{\delta}, t_{\delta})$. 
In the absence of so-called prime modules, the information about $\delta$ is fully captured by
$(T_{\delta}, t_{\delta})$ in the sense that $t_{\delta}(\lca(\{x\}, \{y\})) = \delta(x, y)$ for all $x, y \in X$. 
However, in the presence of prime modules, the full information of $\delta$ cannot be determined 
from $(T_{\delta}, t_{\delta})$ alone.

To summarize at this point, several equivalent notions of symbolic dated maps exist: 
symbolic 2-dissimilarity \cite{Huber2018}, symmetric 2-structures, or edge-colored graphs. 
These structures have in common that they are mainly considered from a tree-reconstruction point of view, 
i.e., characterizations are provided for those symbolic dated maps that can be represented by labeled phylogenetic
trees. In this paper, we aim to generalize the latter concept and ask under which conditions 
symbolic dated maps can be represented by labeled networks, i.e., rooted directed acyclic graphs. 
A first step in this direction is provided in \cite{HS:22,HS:24,HS:24B}, where ``binary'' maps
$\delta\colon X \times X \to \{0,1\}$ have been characterized that can be explained 
by so-called galled-trees, which results in the definition of so-called $\gatex$ graphs.
Furthermore, in \cite{huber2023shared} those symbolic dated maps are characterized
that can be explained by particular 
rather ``tree-like'' directed acyclic graphs which may admit several roots. 
In \cite{BHS:21} it was shown that every symbolic map can be explained by  some rooted median network
and in \cite{Huber2018}  some results are presented for so-called 3-way symbolic maps that arise from restricted versions of galled-trees.

To cover the different, yet similarly flavored, definitions related to symbolic dated maps
and to simplify the notion, we consider \emph{strudigrams}, which are triples $\SM = (X, \Upsilon, \sigma)$
such that $\sigma$ is a map that assigns a unique label 
$\sigma(x,y) \in \Upsilon$ to each two elements $x, y\in X$. 
The word \emph{strudigram} is composed of the first parts of the words \textbf{stru}ctures,
\textbf{di}ssimilarities, and \textbf{gra}phs. A strudigram $\SM$ is explained by a labeled
network $(N, t)$ with leaf set $X$ if,  for all $x, y \in X$, there is a unique least common ancestor $\lca(x, y)$
and it holds that the labels $t(\lca(x, y))$ and $\sigma(x, y)$ coincide.

We provide in Section \ref{sec:prelim} the necessary definitions concerning clustering systems and
networks, and in Section \ref{sec:strudi-N} more details on the notion of strudigrams. As it turns
out, each strudigram $\SM = (X, \Upsilon, \sigma)$ can be explained by some labeled network having
$O(|X|^2)$ vertices, assigning to each two-element subset $\{x, y\}$ of $X$ a unique $\lca$. This
begs the question of whether there are simpler networks. One way to address this question is by
using the modular decomposition tree $(\MDT, \tau)$ of $\SM$ as explained in Section
\ref{sec:moddec}. In the absence of so-called prime vertices in $(\MDT, \tau)$, this labeled tree
explains $\SM$. This is precisely the case if the translation of $\SM$ into a symbolic dated map
$\delta$ is a symbolic ultrametric. The aim is thus to reconstruct a network by locally replacing such prime
vertices by some networks and keeping as much of the tree structure of $\MDT$ as possible. The latter is covered in
Section \ref{sec:pvr}. In particular, we show how to obtain so-called lca-networks to explain
given strudigrams.
While the concept of replacing prime vertices in $(\MDT, \tau)$
is rather general, we focus in Section \ref{sec:gatex} on those networks that can be obtained from
$(\MDT, \tau)$ by replacing prime vertices with so-called galls that are, roughly speaking,
biconnected directed acyclic graphs composed of two paths that only share their start and end point.
Such networks are also known as galled-trees. We provide characterizations of strudigrams $\SM$,
called \gatex strudigrams, that can be explained by such labeled galled-trees in terms of their
underlying so-called primitive substructures. As it turns out, \gatex strudigrams, as well as
networks that explain them, can be constructed in polynomial time, as shown in Section
\ref{sec:algo}.


\section{Preliminaries}
\label{sec:prelim}

\noindent
\textbf{Sets and Clustering Systems.}
All sets considered here are assumed to be finite. For a nonempty set $X$, $\binom{X}{2}$ denotes
the collection of all two-element subsets of $X$ and $2^X$ the powerset of $X$. A \emph{partition of
a set $X$} is a family of nonempty sets $\{X_1,\ldots,X_k\}$ such that $X=\cup_{i=1}^k X_i$ and
$X_i\cap X_j=\emptyset$ for each $i\neq j$. Two sets $X$ and $Y$ \emph{overlap} if $X\cap Y\notin\{\emptyset,X,Y\}$. 

A set system $\mathfrak{S}\subseteq 2^X$ is a
\emph{clustering system (on $X$)} if $\emptyset\notin\mathfrak{S}$, $\{x\}\in\mathfrak{S}$ for all $x\in X$
and $X\in \mathfrak{S}$. For a given parameter $k$, a set system $\mathfrak{Q}$ on $X$ is \emph{pre-$k$-ary} if, 
for all non-empty subsets $A\subseteq X$ with $|A|\leq k$ there is a unique inclusion-minimal element $C\in \mathfrak{Q}$ such that $A\subseteq  C$.
Let $\mathfrak{C}$ be a clustering system. Then, $\mathfrak{C}$ is a \emph{hierarchy} if it does not contain any overlapping clusters.
Furthermore, $\mathfrak{C}$ is \emph{closed} if and only if $A, B \in \mathfrak{C}$  and $A \cap B \neq \emptyset$ implies $A \cap B \in \mathfrak{C}$ (cf.\ \cite[L.\ 16]{Hellmuth2023}). 
Moreover, $\mathfrak{C}$ satisfies property (L) if
  $C_1\cap C_2=C_1\cap C_3$ for all $C_1,C_2,C_3\in\mathfrak{C}$ where $C_1$
  overlaps both $C_2$ and $C_3$.
 Finally,  $\mathfrak{C}$ satisfies property (N3O) if
 $\mathfrak{C}$ contains no three distinct pairwise
    overlapping clusters.

Given a set system $\mathfrak C$ on $X$, we define
\[\mathfrak{C}-x \coloneqq \{C\setminus \{x\} \mid C\in \mathfrak{C} \text{ and } C\neq \{x\}\}.\]

In simple words, $\mathfrak{C}-x$ is the set system obtained from $\mathfrak{C}$ by removing 	$\{x\}$  and  
from each $C\in \mathfrak{C}$ with $C\neq\{x\}$ the element $x$ (if existent). Note that, by definition, 
$\mathfrak{C}-x$ is not a multiset, that is, if $C'\setminus \{x\} = C''$ for some $C',C''\in \mathfrak{C}$
then only one copy of  $C'\setminus \{x\} $ and $ C''$  remains in $\mathfrak{C}-x$.
As the next result shows, the property of being a clustering system that is  closed or that satisfies (L) or (N30)
is heritable.
\begin{lemma}
	\label{lem:hereditary-cluster}
	Let $\mathfrak{C}$ be a clustering system on $X$ with $|X|>1$ that satisfies property $\Pi\in \{closed, \text{(L), (N3O)}\}$
	and let $x\in X$. Then, $\mathfrak{C}-x$ is a clustering system  
	on $X\setminus \{x\}$ that satisfies $\Pi$. 
\end{lemma}
\begin{proof}
	Put $\mathfrak{C}' = \mathfrak{C}-x$ and  $X' \coloneqq X\setminus \{x\}$ for some $x\in X$ with $|X|>1$.
	By construction,  $\mathfrak{C}'$ is a set system. 
	Since $\mathfrak{C}$ is a clustering
	system on $X$ we have, by definition, $\emptyset\notin \mathfrak{C}$, $X\in \mathfrak{C}$ and
	$\{w\}\in \mathfrak{C}$ for all $w\in X$. Since $|X|>1$, we have $X'\neq \{x\}$ and by
	construction, $X'\in \mathfrak{C}'$. Moreover, by construction, $\mathfrak{C}'$ still contains
	$\{w\}\in \mathfrak{C}$ for all $w \in X'$. In particular, $\emptyset\notin \mathfrak{C}'$. Thus,
	$\mathfrak{C}'$ is a clustering system on $X'$.

	Suppose that $\mathfrak{C}$ is closed. Let $A,B \in \mathfrak{C}$ and put $A'=A\setminus \{x\}$
	and $B'=B\setminus \{x\}$. Note that $A=A'$ or $B=B'$ may be possible. Let $C=A\cap B$ and, therefore, 
	$C'=C\setminus \{x\} = A'\cap B'$. We consider now the three cases that may occur for $C$:
	(i) $C=\emptyset$, 	(ii) $C=\{x\}$, and (iii) $C\neq \emptyset$ and $C\neq \{x\}$. 
	If $C = \emptyset$,	then  $C' = A' \cap B' = \emptyset$ and there is nothing to show. 
	If $C = \{x\}$, then $C' = \emptyset$ and again, there is nothing to	show. 
	Assume that $C\neq \emptyset$ and $C \neq \{x\}$. In this case, $C' = A'\cap B'\neq
	\emptyset$. Since $\mathfrak{C}$ is closed, $C \in \mathfrak{C}$ and, by construction of $\mathfrak{C}'$, we have 
	$C'\in	\mathfrak{C}'$. Thus,
	$\mathfrak{C}'$ is closed. 
	
	Suppose that $\mathfrak{C}$ satisfies (L). Again, for a cluster $C_i\in \mathfrak{C}$ that is
	distinct from $\{x\}$, we put $C'_i \coloneqq C_i\setminus \{x\}$. Note that $C_i\neq C_j$ if 
	$C'_i\neq C'_j$. Let $C'_1 \in \mathfrak{C}'$ and assume that $C'_1$ overlaps both $C'_2 , C'_3
	\in \mathfrak{C}'$. One easily verifies that $C_1$ must overlap both $C_2$ and $C_3$. In particular,
	since $\mathfrak{C}$ satisfies (L), we have $C_1\cap C_2=C_1\cap C_3 = A$
	and thus, 	$C'_1\cap C'_2=A\setminus
	\{x\} = C'_1\cap C'_3$. Therefore, $\mathfrak{C}'$ satisfies (L). 
	
	Suppose that $\mathfrak{C}$ satisfies (N3O). If
   $A'=A\setminus \{x\}, B'=B\setminus \{x\}, C'=C\setminus \{x\} \in \mathfrak{C}'$
	overlap pairwise, one easily verifies that $A,B,C$ overlap pairwise;  a contradiction to the 
	assumption that $\mathfrak{C}$ does not contain three distinct pairwise
   overlapping clusters. Hence, $\mathfrak{C}'$ satisfies (N3O). 
\end{proof}

\smallskip\noindent
\textbf{Graphs.}
A \emph{directed graph} $G=(V,E)$ is an ordered pair consisting of a nonempty set $V(G)\coloneqq V$
of \emph{vertices} and a set $E(G)\coloneqq E \subseteq\left(V\times V\right)\setminus\{(v,v) \mid
v\in V\}$ of \emph{edges}. If $(x,y)\in E$ implies  $(y,x)\in E$ for all $x,y \in V$, then
$G$ is called an \emph{(undirected) graph}. Two vertices $x$ and $y$ are \emph{adjacent} if 
$(x,y)\in E$  or $(y,x)\in E$.

A directed graph $G'=(V',E')$ is a \emph{subgraph} of $G=(V,E)$ if
$V'\subseteq V$ and $E'\subseteq E$. For $W\subseteq V$, the \emph{induced subgraph} $G[W]$ of $G$
is the graph with vertex set $V(G[W])=W$ and edge set $E(G[W])=\{(x,y)\in E\mid x,y\in W\}$. In
particular, $G-v$ for $v\in V$ denotes the induced subgraph $G[V\setminus\{v\}]$, whenever $|V|>1$. 

For directed graphs $G=(V,E)$ and $G'=(V', E')$, an \emph{isomorphism
between $G$ and $G'$} is a bijective map $\varphi\colon V\to V'$ such that $(u,v)\in E$ if and only
if $(\varphi(u),\varphi(v))\in E'$. If such a map exists, then $G$ and $G'$ are \emph{isomorphic}, in
symbols $G\simeq G'$. 

A \emph{path} $P_n = (V,E)$ has vertex set $V = \{v_1,v_2,\ldots,v_n\}$ and the edges in $E$ are
precisely of one of the form $(v_i,v_{i+1})$ or $(v_{i+1},v_i)$, $i=1,2,\ldots, n-1$. We often
denote such paths by $P_n=v_1v_2\ldots v_n$ and call them a $v_1v_n$-path. If all edges in $P_n$ are
precisely of the form $(v_i,v_{i+1})$ for each $i=1,2,\ldots, n-1$, then the path $P_n$ is called
\emph{directed} and, otherwise, \emph{undirected}. A path $P_n$ has \emph{length} $n-1$, and the
vertices $v_1$ and $v_n$ are called its \emph{end vertices}. 
In particular, $n=1$ means that the path in
question consists of a single (end) vertex, so that it has length zero. 
A directed graph $G$ is \emph{connected} if there exists
an undirected $xy$-path between any pair of vertices $x$ and $y$. If $G$ has only one vertex, or if
$G-v$ is connected for each vertex $v$ of $G$, then $G$ is said to be \emph{biconnected}. A
\emph{biconnected component} $C=(V',E')$ of $G=(V,E)$ is an inclusion-maximal biconnected induced
subgraph. A biconnected component is \emph{non-trivial} if it has at least three vertices, otherwise
it is \emph{trivial}.

For each $v\in V$ we define $\indeg(v)\coloneqq \left|\left\{u\in V \colon (u,v)\in
E\right\}\right|\quad$ and $\quad \outdeg(v)\coloneqq \left|\left\{u \in V\colon (v,u)\in
E\right\}\right|$ as the \emph{in-degree} respectively \emph{out-degree} of $v$. 
Moreover, $G$ is said to be \emph{acyclic} if it contains no \emph{directed cycle}, that is, no sequence of $k\geq 2$
distinct vertices $v_1,v_2,\ldots, v_k\in V$ such that
$(v_1,v_2),(v_2,v_3),\ldots,(v_{k-1},v_{k}),(v_k,v_1)\in E$.

\smallskip\noindent
\textbf{DAGs and Networks.}
A directed acyclic graph is called \emph{DAG}. 
A vertex of a DAG $G$ with out-degree zero
is called a \emph{leaf}, and the set of leaves in $G$ is denoted by $L(G)$.
In this case, we also say that $G$ is a DAG on $L(G)$.

A \emph{(rooted) network} $N$ is a DAG
containing a unique vertex with in-degree zero. This vertex is called the \emph{root} of $N$, and is
denoted by $\rho_N$. The existence of $\rho_N$ in particular implies that $N$ is connected and
that, for all $v\in V(N)$, there is a directed $\rho_Nv$-path. If $X=L(N)$ we also
state that $N$ is a network \emph{on $X$}. A vertex which is not a leaf is called an
\emph{inner-vertex} and we put $V^0(N)\coloneqq V(N)\setminus L(N)$. In particular, the
\emph{trivial} network $(\{\rho_N\},\emptyset)$ contains a single leaf. Furthermore, a vertex $v\in
V(N)$ is a \emph{tree-vertex} if $\indeg(v)\leq1$ and a \emph{hybrid-vertex} if $\indeg(v)\geq2$. If
a network $N$ contains no hybrid-vertices it is a \emph{tree}. If each leaf of $N$ is a tree-vertex,
then $N$ is \emph{leaf-separated}. 
A network $N'$ is the \emph{leaf-extended version} of a network $N$, if
$N'$ is obtained from $N$ by relabeling all hybrids $x$ in $N$ with $\outdeg_N(x) = 0$ by $v_x$ and adding the edge $(v_x, x)$.
Clearly, if $N$ is a network on $X$ then its leaf-extended version is a network on $X$.
A network on $X$ is \emph{phylogenetic}, if it is either (i) trivial
or (ii) non-trivial and there is no vertex $v$ with $\outdeg(v)=1$ and $\indeg(v)\leq 1$. We
emphasize that, by definition, every non-trivial phylogenetic network $N$ satisfies
$\outdeg(\rho_N)>1$. 

As usual in rooted DAGs $N=(V,E)$, we say that $v$ is a \emph{child} of $u$ while $u$ is a
\emph{parent} of $v$, for each pair of vertices $u$ and $v$ such that $(u,v)\in E$. We denote with
$\parent_N(u)$ and $\child_N(u)$ the set of parents and children of vertex $u$. If $x\in\child_N(u)$
is a leaf, we call $x$ also a \emph{leaf-child} of $u$. If there is a directed path from $u$ to $v$
(possibly of length zero) in $N$, then $u$ is an \emph{ancestor} of $v$ while $v$ is a
\emph{descendant} of $u$ and we write $u\succeq_N v$. If $u\succeq_N v$ but $u\neq v$, we write
$u\succ_N v$. In particular, $\succeq_N$ is a partial order on $V$. We thus say that $u,v \in V$ are
\emph{comparable} if $u\succeq_N v$ or $v\succeq_N u$, and \emph{incomparable} otherwise.  
We sometimes use $v\preceq_N u$ and $v\prec_N u$ instead of
$u\succeq_N v$ and $u\succ_N v$, respectively.

A vertex $v$ in a network $N$ with $\outdeg_N(v)=\indeg_N(v)=1$ is \emph{suppressed} by applying the 
following operations: remove $v$ and its two incident edges $(u,v)$ and $(v,w)$ and add 
the edge $(u,w)$ in case it does not already exist. Note that suppression of a vertex $v$
does not change in- and out-degrees of any of the vertices $z\neq v$.
Given a vertex $v$ in $N$, the \emph{subnetwork rooted at $v$}, denoted $N(v)$, is obtained from the
directed graph $N[\{z\mid z\in V(N),\, v\succeq_N z	\}]$ by suppressing  any of its vertices $w$ for which
$\indeg(w)=\outdeg(w)=1$ and, if its new root $v$ has out-degree one, by deleting the vertex $v$ and
its incident edge, so that the new root of $N(v)$ is the child of $v$.

A vertex $v$ of a network $N$ on $X$ is a \emph{least common ancestor} (LCA) of a non-empty subset
$A\subseteq V(N)$ if it is a $\preceq_N$-minimal element among the set of common ancestors of the
vertices in $A$. The set $\LCA_N(A)$ comprises all LCAs of $A$ in $N$. Note $\LCA_N(A)$ is
non-empty, since $\rho_N$ is a common ancestor of every vertex in $N$. We will, in particular, be
interested in situations where $|\LCA_N(A)|=1$ holds for $A\subseteq X$ (where, usually, $|A|=2$).
For simplicity, we will write $\lca_N(A)=v$ instead of $\LCA_N(A)=\{v\}$ whenever $|\lca_N(A)|=1$
and say that \emph{$\lca_N(A)$ is well-defined}; otherwise, we leave $\lca_N(A)$ \emph{undefined}.
Moreover, we write $\lca_N(x,y)$ instead of $\lca_N(\{x,y\})$. 

For every $v\in V$ in a network $N=(V,E)$ on $X$, the set of its descendant leaves
$\CC_N(v)\coloneqq\{ x\in X\mid x \preceq_N v\}$ is a \emph{cluster} of $N$. We write
$\mathfrak{C}_N\coloneqq\{\CC_N(v)\mid v\in V\}$. In particular, we have
$\emptyset\notin\mathfrak{C}_N$, $\CC_N(x)=\{x\}\in\mathfrak{C}_N$ for each $x\in X$, and since
$\rho_N$ is an ancestor of every vertex in $N$, $\CC(\rho_N)=X\in \mathfrak{C}_N$ holds. Thus
$\mathfrak{C}_N$ is a clustering system for every network $N$. To each clustering system $\mathfrak{C}\subseteq 2^X$ one can
associate a ``canonical'' network, namely its
\emph{Hasse diagram} $\Hasse(\mathfrak{C})$, that is, the DAG with vertex set $\mathfrak{C}$ and directed
edges from $A\in\mathfrak{C}$ to $B\in\mathfrak{C}$ if (i) $B\subsetneq A$ and (ii) there is no
$C\in\mathfrak{C}$ with $B\subsetneq C\subsetneq A$. Since $X$ is an element of $\mathfrak{C}$,
$\Hasse(\mathfrak{C})$ has a unique root. Moreover, the leaves of $\Hasse(\mathfrak{C})$ are
precisely the singleton sets $\{x\}$ for $x\in X$. 
This, however,  implies that 
$\mathfrak{C}_{\Hasse(\mathfrak{C})}\neq \mathfrak{C}$. To circumvent this, we write $G\doteq
\Hasse(\mathfrak{C})$ for the network that is obtained from
$\Hasse(\mathfrak{C})$ by relabeling all vertices $\{x\}$ in
$\Hasse(\mathfrak{C})$ by $x$.  Thus, for $G\doteq \Hasse(\mathfrak{C})$ it
holds that $\mathfrak{C}_G=\mathfrak{C}$ provided that $\mathfrak{C}$ is a
clustering system. 

For later reference, we provide 

\begin{observation}[{\cite[Obs.\ 2.5]{SCHS:24}}] 
  \label{obs:deflcaY}
  Let $G$ be a DAG on $X$, $\emptyset\ne A\subseteq X$, and
  suppose $\lca_G(A)$ is well-defined. Then,  the following is satisfied: \smallskip
  \begin{enumerate}
  \item[(i)] $\lca_G(A)\preceq_{G} v$ for all $v$ with $A\subseteq\CC_G(v)$.
  \item[(ii)] $\CC_G(\lca_G(A))$ is the unique inclusion-minimal cluster in
    $\mathfrak{C}_G$ containing $A$.
  \item[(iii)] $\lca	_G(\CC_G(\lca_G(A)))=\lca_G(A)$. 
  \end{enumerate} \smallskip
\end{observation}

\begin{definition}
A DAG $G$ on $X$ has the \emph{$\klca$-property} if $\lca_G(A)$ is well-defined for all non-empty subsets $A\subseteq X$ with $|A|\leq k$, 
Moreover, $G$ is an \emph{lca-network} if it has the $\klca$-property for all $k\leq |X|$.
A DAG $G$ is a \emph{global} lca-network, if $\lca_G(A)$ is well-defined for all non-empty subsets $A\subseteq V(G)$.
\end{definition}

\begin{figure}[t]
  \begin{center}
        \includegraphics[width=0.15\textwidth]{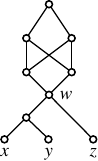}
          \caption{Shown is an lca-network $N$ on $X = \{x,y,z\}$ with pre-$|X|$-ary clustering
          system $\mathfrak{C}_N = \{\{x\},\{y\},\{x\},\{x,y\}, \{x,y,z\}\}$. None of the vertices
          $v$ in $N$ with $v\succ_N w$ serve as the $\lca$ for any two leaves.}
          \label{fig:lca-N}
  \end{center}
\end{figure}

\begin{theorem}[{\cite[L.\ 40 \& Prop.\ 11]{Hellmuth2023}}]\label{thm:Chat-lcaN}
A clustering system $C$ is closed if and only
if there is an lca-network $N$ with $\mathfrak{C} = \mathfrak{C}_N$. 
In this case,  $G\doteq \hasse(\mathfrak{C})$ is an lca-network.
Moreover, every rooted tree is an lca-network.
\end{theorem}

We refer to \cite{SCHS:24} for more details on the relationship between
DAGs or networks with $\klca$-property  and their underlying clustering systems. 
An example of a lca-network is provided in Fig.~\ref{fig:lca-N}.
By Obs.\ \ref{obs:deflcaY}(ii), $\CC(\lca_G(A))$ is the unique inclusion-minimal cluster in $\mathfrak{C}_G$ containing $A$ for all non-empty subsets $A\subseteq X$ with $|A|\leq k$.
Hence $\mathfrak{C}_N$  is pre-$k$-ary.  
\begin{observation}\label{obs:2lca->pre-bin}
For every DAG $G$ with $\klca$-property, 
the clustering system $\mathfrak{C}_G$ is pre-$k$-ary. 
\end{observation}

We note that in case the network $N$ is understood from context, we drop the subscript from all notation, so that e.g. $\lca_N(x,y)$ becomes $\lca(x,y)$.

\begin{remark}
In all drawings that include networks, we omitted drawing the directions of edges as arcs, that is, edges
are  drawn as simple lines. However, in all cases, the directions are implicitly given by directing the edges
``away'' from the root.
\end{remark}


\section{Strudigrams and Explanations by Labeled Networks}
\label{sec:strudi-N}

We will consider \emph{labeled} networks $(N,t)$, that is, networks $N=(V,E)$ equipped with a
(vertex-)\emph{labeling} $t\colon V^0(N)\to\Upsilon$. If $N$ is a network on $X$ that has the
$\2lca$-property, then we can derive from $(N,t)$ the ordered triple $\mathcal{S}(N,t) =
(X,\Upsilon,\sigma)$, where $\Upsilon$ is the image of $t$ and $\sigma\colon \binom{X}{2}\to
\Upsilon$ is the map defined by $\sigma(\{x,y\})=t(\lca(x,y))$ for all distinct $x,y\in X$.

An undirected graph $G=(X,E)$ is a \emph{cograph}, if there is a labeled tree $(T,t)$ on $X$, called
\emph{cotree}, with $t\colon V^0(T)\to \{0,1\}$ such that $x$ and $y$ are adjacent in $G$ if and only if
$t(\lca(x,y))=1$ for all distinct $x,y\in X$ \cite{Corneil:85}. In other words, the structure of a
cograph $G$ is entirely captured by its cotree $(T,t)$. One easily observe that there is direct
translation of any cograph $G=(X,E)$ into a triple $\SM_G= (X,\{0,1\},\sigma)$ where $\sigma\colon
\binom{X}{2}\to \{0,1\}$ such that $\sigma(\{x,y\})=1$ precisely if $\{x,y\}\in E$. Hence, $G$ is a
cograph if and only if $\SM_G$ can be ``explained'' by a labeled tree $(T,t)$, i.e., $\SM_G =
\mathcal{S}(T,t)$. 
As already outlined in the introduction, there is a natural generalization of the latter 
concept by means of symbolic dated maps. 
In particular, a symbolic dated map can be  explained by a labeled tree $(T,t)$ precisely
if it is a symbolic ultrametric \cite{BD98}. Our aim is to generalize these 
results and to consider labeled networks rather than trees. 
To achieve this goal, we
introduce a more universal structure: \emph{strudigrams} which are triples $\SM = (X, \Upsilon,
\sigma)$, where $V(\SM)\coloneqq X$ and $\Upsilon(\SM)\coloneqq \Upsilon$ are non-empty finite sets
and where $\sigma\colon \binom{X}{2} \to \Upsilon$ is a map that assigns a unique label
$\sigma(\{x,y\})\in \Upsilon$ to each two-element subset $\{x,y\}$ of $X$. The elements in $X$ are
called \emph{vertices} of $\SM$, and we call $|X|$ the \emph{size} of $\SM$. For simplicity, we put
$\sigma(xy) = \sigma(yx) \coloneqq \sigma(\{x,y\})$. 

\begin{definition}\label{def:strudi-explain}
A strudigram  $\SM = (X,\Upsilon,\sigma)$ is \emph{explained by a network $(N,t)$} if 
$N$ has the $\2lca$-property and $\SM = \mathcal{S}(N,t)$. 
\end{definition}

\begin{figure}
	\centering
	\includegraphics[width=.7\textwidth]{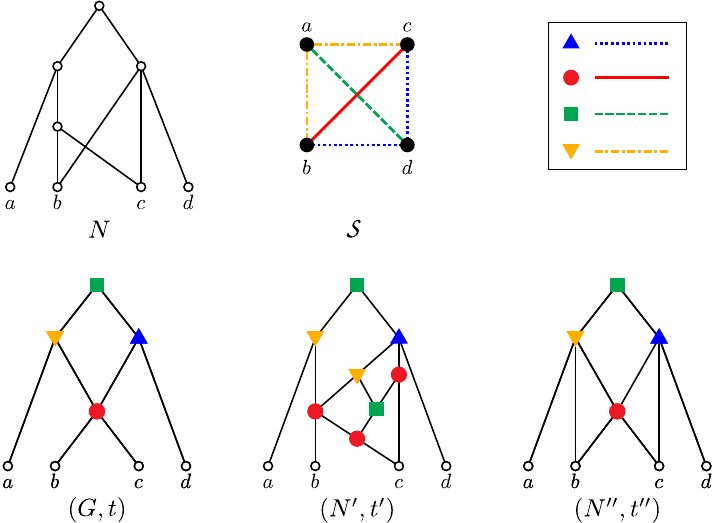}
	\caption{The network $N$ has clustering system $\mathfrak{C}_N = \{\{a\},\{b\},\{c\},\{d\},\{b,c\},\{a,b,c\},\{b,c,d\},X\}$
				on $X  = \{a,b,c,d\}$. Since $N$  does not have the $\2lca$-property, it does not  explain any strudigram 
					according to Def.\ \ref{def:strudi-explain}. Moreover, a strudigram $\SM$ on $X  = \{a,b,c,d\}$ is depicted as an edge-colored graph.
					It contains an rainbow-triangle and, by 	Theorem \ref{thm:char-treeEx}, it cannot be explained by a labeled tree. 
					The three networks $G\doteq \hasse(\mathfrak{C}_N)$, $N'$ and $N''$ have the same clustering system $\mathfrak{C}_N$
					and satisfy the $\2lca$-property. 
					In particular, the labeled versions  $(G,t)$, $(N',t')$ and $(N'',t'')$  all explain $\SM$, with edge-colors corresponding to the respective label as shown in the legend. 
					Note that $G$ is a galled-tree while $N'$ and $N''$ are not.
				} 
	\label{fig:exmpl-sketch1}
\end{figure}

For an illustrative example of Definition \ref{def:strudi-explain}, see Figure \ref{fig:exmpl-sketch1}.
As already outlined in introduction, the word \emph{strudigram} is composed of the first parts of the words \textbf{stru}ctures,
\textbf{di}ssimilarites and \textbf{gra}phs to cover the fact that strudigrams
can be used to represent different, yet similarly flavored combinatorial objects. Each strudigram
$\SM = (X, \Upsilon, \sigma)$ gives, for each subset $\Upsilon'\subseteq \Upsilon$, rise to a unique 
edge-colored undirected graph $G = (X, E)$ with $E\subseteq \binom{X}{2}$ and
 a coloring $c\colon E \to \Upsilon'$. In
this case, $E$ represents the set of all $\{x, y\} \in \binom{X}{2}$ for which $\sigma(xy)\in \Upsilon'$
and $c(\{x, y\}) \coloneqq \sigma(xy)$ for all $\{x, y\} \in E$. In a more direct way, $\SM$ yields
an undirected complete edge-colored graph $K_{|X|} = (X, \binom{X}{2})$ with color-map $c = \sigma$, 
structures that have been considered e.g.\ in \cite{Gurvich:09}.

Moreover, $\SM$ also defines a (symmetric) 2-structure $g=(X,\P)$, where
$\P=\{X_i\}_{i\in\Upsilon}$ is a partition of $(X\times X)\setminus \{(v,v)\mid v\in X\}$
such that $(x,y)\in X_i$ precisely if $\sigma(xy)=i$
\cite{Ehren1990A,Ehren1990B,ehrenfeucht1999theory}. Since $(x,y)$ and $(y,x)$ always lie in the same
set of $\P$, $g$ is a symmetric 2-structure . Naturally, $g$ may be extended to a
\textit{labeled} symmetric 2-structure $(X,\P,\Upsilon,\delta)$, where
$\delta:\mathcal{P}\to\Upsilon$ assigns a label to each set in the partition ($\delta(X_i)=i$ would
be a reasonable choice) \cite{Engelfriet1996}. Despite its symmetry, the 2-structure operates on
elements of $X\times X$ instead of $\binom{X}{2}$, distinguishing it from strudigrams.

Furthermore, $\SM$ defines symmetric symbolic maps $\delta\colon X\times X \to \Upsilon$ as well as 2- dissimilarities $\delta\colon X\cup \binom{X}{2} \to \Upsilon$ determined by $\delta(x, y) = \sigma(xy)$ for distinct $x,y\in X$ and by putting $\delta(x,x) = \odot$, resp., $\delta(x) = \odot$
for some symbol $\odot \notin \Upsilon$ \cite{HSW:17,BD98,Hellmuth:13a,Huber2018}. Conversely, each
edge-colored undirected (complete) graph, symmetric 2-structure, or symmetric symbolic dating map and
2-dissimilarity defines a strudigram in a natural way. In other words, strudigrams help represent
classical combinatorial objects, and the results established here for strudigrams hold, therefore,
for these objects as well. Vice versa, results that have been established for edge-colored
(complete) graphs, symmetric 2-structures, or symmetric symbolic dating maps hold for strudigrams.

A strudigram $(X,\Upsilon,\sigma)$ contains an induced $P_4$ if there are four vertices $a,b,c,d\in X$
such that $k\coloneqq \sigma(ab)=\sigma(bc)=\sigma(cd)$ and $k\notin \{\sigma(ac), \sigma(ad),\sigma(bd)\}$. 
$P_4$-free strudigrams are strudigrams without induced $P_4$s. 
A strudigram $(X,\Upsilon,\sigma)$ contains a rainbow-triangle, if there are three vertices $a,b,c\in X$
such that $\sigma(ab)$, $\sigma(ac)$, $\sigma(bc)$ are pairwise distinct. 
Rainbow-triangle-free strudigrams are strudigrams without rainbow-triangles. 
Note that rainbow-triangle-free and $P_4$-free strudigrams are also known as
symbolic ultrametrics (cf.\ \cite[Prop.\ 3]{Hellmuth:13a})
or, equivalently,  as symmetric uniformly non-primitive 2-structures \cite{HSW:17,Engelfriet1996}. 

\begin{theorem}[{\cite[Prop.\ 1 \& 3]{Hellmuth:13a}}]\label{thm:char-treeEx}
A strudigram can be explained by a labeled tree if and only if 
it is rainbow-triangle-free and $P_4$-free. 
\end{theorem}

By Theorem \ref{thm:char-treeEx},  plenty of strudigrams exist that cannot be explained by any labeled tree. 
This motivates the question as whether there are labeled networks that can explain any given strudigram.
The next result provides an affirmative answer. 
\begin{proposition}\label{prop:SM-hasExpl}
	For every strudigram $\SM = (X,\Upsilon,\sigma)$ there is a phylogenetic lca-network $(N,t)$ on $X$ with
	$O(|X|^2)$ vertices that explains it. 
\end{proposition}
\begin{proof}
	Let $\SM = (X,\Upsilon,\sigma)$ be a strudigram and consider the set system $\mathfrak{Q}
	\coloneqq \{\{x,y\}\mid \text{not necessarily distinct } x,y\in X\} \cup \{X\}$. By construction,
	$\mathfrak{Q}$ is a clustering system. Consider the network $N\doteq \Hasse(\mathfrak{Q})$ that
	is obtained from the Hasse diagram $\Hasse(\mathfrak{Q})$ by relabeling all vertices $\{x\}$ with
	$x$, ensuring that $L(N)=X$. This together with \cite[Lem.~22]{Hellmuth2023} implies that $N$ is
	a phylogenetic network on $X$. By construction, $N$ contains $1+|X|+\frac{|X|(|X|-1)}{2} \in
	O(|X|^2)$ vertices. One easily observes that $\mathfrak{Q}$ is closed. 
	Theorem \ref{thm:Chat-lcaN} implies that $N$ is an lca-network. 
	In particular, for any vertices $x,y\in X$, the set $\{x,y\}$ is the unique
	inclusion-minimal cluster in $\mathfrak{Q}$ containing $x$ and $y$.
	By Obs.~\ref{obs:deflcaY}, we have $\lca_{N}(x,y)=v$ where $v$
	is the unique vertex in $N$ that satisfies $\CC(v)=\{x,y\}$. This, in turn, implies that
	for all distinct 2-element subsets $\{x,y\}$ and $\{a,b\}$ it holds that $\lca_{N}(x,y)\neq
	\lca_{N}(a,b)$. We can now equip $N$ with a labeling $t$ defined by $t(\lca_{N}(x,y))\coloneqq
	\sigma(xy)$ for all $x,y\in X$ and $t(\rho_{N})\coloneqq \upsilon$ for some
	$\upsilon\in\Upsilon$. By the latter arguments, $(N, t)$ is a phylogenetic lca-network that explains
	$\SM$. 	
\end{proof}
We note that it has already been shown in \cite{BHS:21} that every strudigram can be explained
by so-called labeled half-grids which also contain  $O(|X|^2)$ vertices. 
Note also that if $\sigma\colon \binom{X}{2} \to \Upsilon \coloneqq \{1,2,\dots, \frac{|X|(|X|-1)}{2}\}$
is surjective, then $O(|X|^2)$ vertices are necessary in any network that explain $\SM = (X,\Upsilon,\sigma)$. 

	Essential for networks explaining strudigrams is the requirement that they satisfy the $\2lca$-property.
	By Proposition \ref{prop:SM-hasExpl}, there is even always an lca-network that explains a strudigram. However,
    the network constructed in the proof of Proposition \ref{prop:SM-hasExpl} is rather ``dense'', and simpler
    constructions of lca-networks may exist. A nice feature of lca-networks $(N,t)$ explaining $\SM$
    is that labeled lca-``sub''networks $(N',t')$ can be constructed from $(N,t)$ to explain
    substrudigrams of $\SM$.
	 As we shall see, this result will become quite handy later in Section \ref{sec:gatex}. 

\begin{proposition}\label{prop:lcaN-explain}
	Let  $(N,t)$ be a labeled lca-network that explains $\SM$ and $W\subseteq V(\SM)$ be a non-empty subset. 
    Put $\mathfrak{C}'\coloneqq \{C \mid C\in \mathfrak{C}_N \text{ and } C \subseteq  W\}$. 
    Then, $\mathfrak{C}'$ is a clustering system on $W$ and 
	$G\doteq \Hasse(\mathfrak{C}')$ is an lca-network that can be equipped with a labeling
	$t'$ such that $(G,t')$ explains $S[W]$.	
\end{proposition}
\begin{proof}
	Let $\SM = (X,\Upsilon,\sigma)$ be a strudigram that is explained by
	a labeled lca-network $(N,t)$ on $X$, and let $W \subseteq X$ be a non-empty set. 
	If $|X|=1$, then $W=X$ must hold. In this case, $G$ is the single vertex network
	and $(G,t')$ trivially explains $\SM[W]=\SM$ for any labeling $t'$. Hence, assume that $|X|>1$. 
	Since any non-empty subset $W\subseteq V(\SM)$ can be obtained from $X$ by
	removing vertices in $X\setminus W$
	one by one and since, by Prop.~\ref{prop:SM-hasExpl}, every (sub)strudigram $\SM[W]$
	can be explained by some labeled lca-networks, it is sufficient to show that 
	the statement is true for $W\coloneqq X\setminus\{x\}$ and  some $x\in X$. 
	
	Put $\mathfrak{C} =\mathfrak{C}_N$ and $\mathfrak{C}' \coloneqq \mathfrak{C}-x$. 
	By Lemma~\ref{lem:hereditary-cluster}, $\mathfrak{C}'$ is a clustering system on $W$. 
	Since 	$N$ is an lca-network,  Theorem \ref{thm:Chat-lcaN} implies that 
	$\mathfrak{C}$ must be closed. By Lemma \ref{lem:hereditary-cluster},
	$\mathfrak{C}'$ must be closed.
	Moreover, Theorem \ref{thm:Chat-lcaN} implies
	that $G\doteq \Hasse(\mathfrak{C}')$ is an lca-network on $W$.
	
	Since $N$ is an lca-network, 
	$\lca_N(\CC_G(v))$ is well-defined in $N$ for each $v\in V^0(G)$. Let $w_v \coloneqq
	\lca_N(\CC_G(v))$ for all $v\in V^0(G)$. Based on the latter,  we can establish a well-defined
	labeling $t'\colon V^0(G)\to \Upsilon$ by putting, for all $v\in V^0(G)$, \[t'(v) = t(w_v).\]  
	It remains to show that $(G,t')$ explains $\SM[W]$. Hence, let $a,b\in W$ be distinct. Since $N$
	has, in particular, the $\2lca$-property, $u\coloneqq \lca_N(a,b)$ is well-defined. Obs.\
	\ref{obs:deflcaY}(ii) together with $u = \lca_N(a,b)$ implies that $\CC_N(u)$ is the unique
	inclusion-minimal cluster in $\mathfrak{C}$ that contains $a,b$. Similarly, 
	$G$ has the $\2lca$-property and thus, $v\coloneqq \lca_G(a,b)$ is well-defined.

	In the following, we show that $\CC_N(u)\setminus \{x\} = \CC_G(v)$. To this end, we show first
	that $C' \coloneqq \CC_N(u)\setminus \{x\} \in \mathfrak{C}'$ is an inclusion-minimal cluster
	in $\mathfrak{C}'$ that contains $a,b$. We consider the two cases $x\notin \CC_N(u)$ and $x\in
	\CC_N(u)$. 

	Suppose that $x\notin \CC_N(u)$ and thus, $C' = \CC_N(u)$. 
	Assume, for contradiction, that there
	is some $C''\in \mathfrak{C}'$ such that $C''\subsetneq C'$ and $a,b\in C''$. Since $C' =
	\CC_N(u)$ is the unique inclusion-minimal cluster in $\mathfrak{C}$ that contains $a,b$, we have
	$C''\notin \mathfrak{C}$. However, since $C''\in \mathfrak{C}'$, it must hold that $C'' \cup
	\{x\} = C$ for some $C\in \mathfrak{C}$. Let $z$ be some vertex with $\CC_N(z)=C$. Since $x\notin
	C'$ and $x\in C$ it follows that $u\neq z$. This together with $u =\lca_N(a,b)$ and $a,b\in C''$
	and Obs.\ \ref{obs:deflcaY}(i) implies that $u\prec_N z$. Hence, $C' \subsetneq C = C'' \cup
	\{x\}$ (cf.\ \cite[L.\ 17]{Hellmuth2023}). 
	Since $x\notin C'$, it holds that $C' \subseteq C''$ 
	a contradiction. Hence, $C' = \CC_N(u)$ is an inclusion-minimal cluster in
	$\mathfrak{C}'$ that contains $a,b$. 
	
	Suppose that $x\in \CC_N(u)$ and thus, $C' \cup \{x\} = \CC_N(u)$. Assume again, for
	contradiction, that there is some $C''\in \mathfrak{C}'$ such that $C''\subsetneq C'$ and $a,b\in
	C''$. Thus, $C''\subsetneq \CC_N(u)$. Since $\CC_N(u)$ is the unique inclusion-minimal cluster in
	$\mathfrak{C}$ that contains $a,b$, we have $C''\notin \mathfrak{C}$. However, since $C''\in
	\mathfrak{C}'$, it must hold that $C'' \cup \{x\} = C$ for some $C\in \mathfrak{C}$. Since
	$x\notin C'$, $x\notin C''$ and $C''\subsetneq C'$, we have $C = C'' \cup \{x\}\subsetneq C'\cup
	\{x\} = \CC_N(u)$; a contradiction to $\CC_N(u)$ being the unique inclusion-minimal cluster in
	$\mathfrak{C}$ that contains $a$ and $b$. Hence, $C'$ is an inclusion-minimal cluster in
	$\mathfrak{C}'$ that contains $a,b$.

	Finally, Obs.\ \ref{obs:deflcaY}(ii) implies that $\CC_G(\lca_G(a,b)) = \CC_G(v)$ is the unique
	inclusion-minimal cluster in $\mathfrak{C}'$ that contains $a,b$. Hence, $C' = \CC_N(u)\setminus
	\{x\} = \CC_G(v)$. 
	
	We are now in the position to show that $\sigma(ab) = t(\lca_N(a,b)) = t'(\lca_G(a,b))$ and, thus
	to prove that $(G,t')$ explains $\SM[W]$. By Obs.\ \ref{obs:deflcaY}(iii), $u = \lca_N(a,b) =
	\lca_N(\CC_N(\lca_N(a,b))) = \lca_N(\CC_N(u))$. Assume that $x\notin \CC_N(u)$. In this case,
	$\CC_N(u) = \CC_G(v)$ and, therefore, $u = \lca_N(\CC_N(u)) = \lca_N(\CC_G(v)) = w_v$. Hence,
	$\sigma(ab) = t(\lca_N(a,b)) = t(w_v) = t'(v) = t'(\lca_G(a,b))$. Assume that $x\in \CC_N(u)$ in
	which case $\CC_G(v) \notin \mathfrak{C}$ since, otherwise $a,b\in \CC_G(v)$ and
	$\CC_G(v)\subsetneq \CC_N(u)$ would contradict the fact that $\CC_N(u)$ is the inclusion-minimal
	cluster in $ \mathfrak{C}$ containing $a,b$. 
	Thus, $\CC_N(u)$ is an inclusion-minimal cluster
	containing $\CC_G(v)$. Since $N$ has the $\klca$-property, Obs.\ \ref{obs:2lca->pre-bin} implies that
	$\mathfrak{C}$ is pre-$k$-ary and thus, $\CC_N(u)$ is the unique inclusion-minimal cluster in
	$\mathfrak{C}$ containing $\CC_G(v) = C' = \CC_N(u)\setminus \{x\}$. 
	Hence, $u = \lca_N(\CC_N(u)) = \lca_N(\CC_N(u)\setminus \{x\}) = \lca_N(\CC_G(v)) = w_v$. Thus,
	$\sigma(ab) = t(\lca_N(a,b)) = t(w_v) = t'(v) = t'(\lca_G(a,b))$. 
	In summary, for the clustering system $\mathfrak{C}'$ on $W$,   
	$G\doteq \Hasse(\mathfrak{C}')$ is an lca-network that can be equipped with a labeling $t'$ such that $(G,t')$ explains $S[W]$.
\end{proof}


\section{Strudigrams and their Modular Decomposition}
\label{sec:moddec}

Modular decomposition and the notion of modules have a long history and goes back
to the seminal work by Gallai \cite{gallai-67} in 1967 for the purpose of
constructing transitive orientations of so-called comparability graphs. Since then this
concept has been rediscovered by various researchers under different names for undirected graphs \cite{Moehring:84, Moh:85,HM-79,Blass:78,HP:10}. 
Generalizations to directed graphs \cite{MCCONNELL05}, 
hypergraphs \cite{bonizzoni1995modular,BONIZZONI199965,HABIB202256},
$k$-structures \cite{EHRENFEUCHT1994209}, sets of (not necessarily disjoint)
binary relations \cite{HSW:17}, Boolean functions
\cite{BIOCH20051}, homogeneous relations \cite{BHLM:09}, $k$-ary relations or
matroids \cite{Moh:85,Moehring:84} have been established. 
We will review here some of these concepts needed for the modular decomposition of strudigrams.

Let $\SM= (X,\Upsilon,\sigma)$ be a strudigram and $W\subseteq X$. Then, $\sigma_{|W}\colon \binom{W}{2} \to
\Upsilon$ denotes the restriction of $\sigma$ to the set $\binom{W}{2}$, i.e., $\sigma_{|W}(xy) = \sigma(xy)$ for all distinct
$x,y\in W$. Moreover, $\SM[W] = (W,\Upsilon,\sigma_{|W})$ is the strudigram \emph{induced by $W$} and
we put $\SM-v\coloneqq \SM[X\setminus \{v\}]$. Given two vertex-disjoint strudigrams $\SM =
(X,\Upsilon,\sigma)$ and $\SM' = (X',\Upsilon', \sigma')$ we can construct a new strudigram
$\SM\join_k \SM' = (X\cup X', \Upsilon\cup \Upsilon'\cup\{k\},\widetilde \sigma)$, called the
\emph{$k$-join} of $\SM$ and $\SM'$ by putting
\[\widetilde\sigma(xy)\coloneqq 
    \begin{cases}
		 \sigma(xy)&\text{if $x,y\in X$}\\
		 \sigma'(xy)&\text{if $x,y\in X'$}\\
	    k&\text{otherwise.}
\end{cases}
\]
Note that $\join_k$ is associative and commutative on vertex-disjoint strudigrams \cite{Engelfriet1996}.
Moreover, for $\join_k$, we may have $k\in \Upsilon\cup \Upsilon'$.
If a strudigram $\SM$ can be written as the $k$-join $\SM = \SM'\join_k \SM''$ we call
$\SM$ a \emph{$k$-series strudigram}. 

Two strudigrams $\SM= (X,\Upsilon,\sigma)$ and $\SM'= (X',\Upsilon',\sigma')$ are \emph{color-restrictive (cr)-isomorphic},
in symbols $\SM\simeq \SM'$, if there is a bijective map $\varphi\colon X\to X'$ such that 
$\sigma(xy)=\sigma'(\varphi(x)\varphi(y))$ for all distinct $x,y\in X$.

Let $\SM = (X,\Upsilon,\sigma)$ be a strudigram. A \emph{module} $M$ of $\SM$ is a non-empty subset
$M\subseteq X$ such that for all $x, y \in M$ and all $z\in X\setminus M$ it holds that $\sigma(xz)
= \sigma(yz)$. One easily verifies that $X$ and all $\{x\}$ with $x\in X$ are modules. These modules are
called \emph{trivial}, and all other modules \emph{non-trivial}. We denote with $\MD(\SM)$ the set
of all modules of $\SM$.

\begin{figure}
	\centering
	\includegraphics[width=0.8\textwidth]{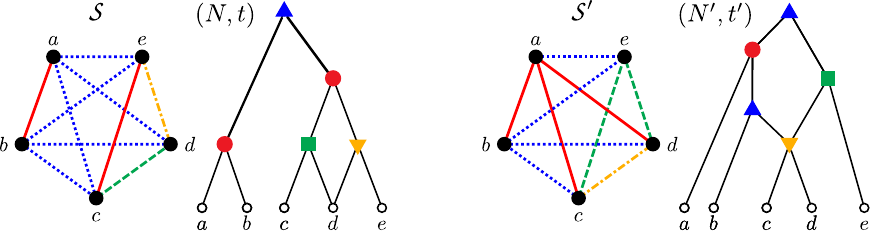}
	\caption{The edge-colored graph representation of a strudigram $\SM$ explained by $(N,t)$
	             and $\SM'$ explained by $(N',t')$, see Fig.~\ref{fig:exmpl-sketch1} for the color legend. $\SM$ is a $k$-series strudigram as
	             $\SM=\SM[\{a,b\}]\join_k\SM[\{c,d,e\}]$ where $k$ refers to the
	             edge-color \textit{dotted-blue} in the drawing. In this example, both $\SM[\{a,b\}]$ and $\SM[\{c,d,e\}]$
	             are $k$-prime. However, $\SM[\{a,b\}]$ is not prime as 
	             $\SM[\{a,b\}]=\SM[\{a\}]\join_{k'}\SM[\{b\}]$ where $k'$ refers to the edge-color \textit{solid-red}. In contrast, $\SM[\{c,d,e\}]$ is
	             prime and, even primitive. Moreover, $\SM'$ is prime. Neither $\SM$ nor $\SM'$ is
	             primitive, since $\{a,b\}$ (resp.\ $\{c,d\}$) is a strong module of $\SM$ (resp.\
	             $\SM'$). 
				} 
	\label{fig:kseries-prime}
\end{figure}

\begin{definition}\label{def:k-series}
A module $M$ of $\SM$ is \emph{$k$-series} if $S[M]$ is $k$-series, that is, 
if there is a $k\in \Upsilon$ and a partition $\{M_1,
M_2\}$ of $M$ such that $\SM[M]=\SM[M_1]\join_k \SM[M_2]$. 
\end{definition}
There are modules that are not $k$-series. A simple example is provided by 
the rainbow-triangle $\SM = (X,\Upsilon,\sigma)$ with $|X| = 3$ and $\sigma(xy)\neq \sigma(ab)$ for all distinct $\{x,y\},\{a,b\}\in \binom{X}{2}$
with $\{x,y\}\neq \{a,b\}$. In this case, for any partition $\{M_1, M_2\}$ of $X$ we have 
$|M_i|>1$ for some $i\in \{1,2\}$ and thus $\sigma(xy)\neq \sigma(xy')$ for
$x\in M_j$ and $y,y'\in M_i$ with $j\in \{1,2\}\setminus\{i\}$.
Thus, $\SM[X]\neq \SM[M_1]\join_k \SM[M_2]$ for all partitions $\{M_1, M_2\}$ of $X$
and therefore, the module $X$ is not $k$-series for any $k\in \Upsilon$. 
Further examples are provided in Fig.\ \ref{fig:kseries-prime}.

\begin{definition}
Modules that are not $k$-series for a given $k$ are called \emph{$k$-prime}.
A module $M$ is \emph{prime} (termed \emph{irreducible} in \cite{Engelfriet1996})
if it is $k$-prime for all $k$. We call a strudigram $\SM$  \emph{$k$-prime}, resp., \emph{prime} if $V(\SM)$ is $k$-prime, resp., prime.
\end{definition}

A module $M$ is \emph{strong} if $M$ does not overlap with any other module and we denote with
$\MDstrong(\SM)$ the set of all strong modules of a strudigram $\SM$, often referred to as its \emph{modular decomposition}. 
Since the trivial modules of $\SM$ are
strong, we have $\MDstrong(\SM)\neq \emptyset$. Moreover, $\MDstrong(\SM)$ is uniquely determined and has size
$O(|V(\SM)|)$ \cite{EHREN1994}, although there may be exponentially many modules. For example, in
the strudigram $\SM = (X,\{k\},\sigma)$ every $M\subseteq X$ is a module and, thus, $\SM$ has
$2^{|X|}$ modules, however, the $|X|+1$ strong modules are $X$ and the singletons $\{x\}$, $x\in X$.
Note that if $M$ is a module of $\SM$ that is not strong, i.e. $M$ overlaps with some module $M'$,
then by \cite[Thm.~4.12]{Ehren1990A} both $M\setminus M'$ and $M\cap M'$ are modules of $\SM$. Since
$M\setminus M'$ and $M\cap M'$ are disjoint modules there is, by \cite[Lem.4.11]{Ehren1990A}, some
$k$ such that $\SM[M]=\SM[M\setminus M']\join_k\SM[M\cap M']$, ensuring that $M$ is not a prime
module. We summarize the latter discussion into
\begin{observation}
Every prime module of a strudigram is a strong module.	
\end{observation}

Note that, for a $k$-series module $M$, the partition $\{M_1, M_2\}$ is not necessarily unique,
i.e., $\SM[M]=\SM[M_1]\join_k \SM[M_2] = \SM[M'_1]\join_k \SM[M'_2]$ for different partitions
$\{M_1, M_2\}$ and $\{M'_1, M'_2\}$ of $M$ may be possible. The simplest example in this context is
again the strudigram $\SM = (X,\{k\},\sigma)$ with $\sigma(xy)=k$ for all $\{x,y\}\in \binom{X}{2}$ for
which $\SM[X]=\SM[M_1]\join_k \SM[M_2]$ for any partition of $\{M_1, M_2\}$ of $X$. Although
partitions of $k$-series strudigrams are not necessarily uniquely determined, there is, for each
non-prime (sub)strudigram, a unique partition into $k$-prime subsets for some unique $k$.

\begin{theorem}[{\cite[Lem.~2.15, Thm.\ 2.16]{Engelfriet1996}}]\label{thm:kseries=k'series}
Let $\SM= (X,\Upsilon,\sigma)$  be a strudigram and $W$ be an arbitrary non-empty subset of $X$.
Then, $\SM[W]$ is either prime or there is a unique $k\in \Upsilon$
and a unique partition $\{W_1,\dots,W_\ell\}$ of $W$ 
such that $\ell\geq 2$, each $\SM[W_i]$ is $k$-prime and $\SM[W] = \SM[W_1]\join_k \cdots\join_k \SM[W_k]$. 
In this case, $\Mmax(\SM[W]) = \{W_1,\dots,W_\ell\}$.
\end{theorem}

Note that, for every strudigram $\SM$ on $X$, the set $X$ and all singletons $\{x\}$ with $x\in X$
are contained in $\MDstrong(\SM)$. This and the fact that 
 $\MDstrong(\SM)$ does not contain pairwise overlapping modules, implies that $\MDstrong(\SM)$ is a
clustering system on $\{\{x\} \mid x\in X\}$ that is a hierarchy. 
Hence, $\MDstrong(\SM)$ is uniquely represented  by the tree $\MDT_{\SM}\coloneqq \Hasse(\MDstrong(\SM))$. The root of $\MDT_{\SM}$
is, by definition, the vertex $X$. We denote with $\MDstrong^0(\SM)$ the set all strong modules that are
distinct from the singletons. In addition, $\MDT_\SM$
is equipped with the following \emph{labeling}
$\tau_\SM:\MDstrong^0(\SM)\to\Upsilon\cup\{\primeL\}$, defined by
\[\tau_\SM(M)=\begin{cases}
    k & \text{if $M$ is $k$-series}\\
    \primeL&\text{if $|M|>1$ and $M$ is prime}
\end{cases}
\]
We write $(\MDT_\SM, \tau_\SM)$ for a modular decomposition tree equipped with such a labeling $\tau_\SM$. 
By Theorem \ref{thm:kseries=k'series}, each module $M$ is either $k$-series for some unique $k\in \Upsilon$ or prime. 
Hence, the labeling $\tau_\SM$ is uniquely determined. 
The pair $(\MDT_\SM, \tau_\SM)$ at least partially explains the underlying strudigram $\SM$, 
in the sense that if $\tau_\SM(\lca_{\MDT_\SM}(\{x\},\{y\}))=k\neq\primeL$ for two distinct $x,y\in X$ then we have $\sigma(xy)=k$. 
If, on the other hand, $\tau_\SM(\lca_{\MDT_\SM}(\{x\},\{y\}))=\primeL$, then $\MDT_\SM$ does not reveal information about $\sigma(xy)$. 

\begin{definition}[\cite{ehrenfeucht1999theory}]
Strudigrams $\SM$ that consist of trivial modules only are called \emph{primitive}. 
Moreover, $\SM$ is \emph{truly-primitive} if it is primitive and it has at least three vertices. 
\end{definition}

Clearly, if $|V(\SM)|=1$, then $\SM$ is primitive. The term truly-primitive is necessary to
distinguish between primitive strudigrams that are at the same time $k$-series (which is precisely the case when $|V(\SM)|=2$)
and those that are not.

\begin{theorem}\label{thm:unp<=>treeExpl}
    For every strudigram $\SM=(X,\Upsilon,\sigma)$, the following statements are equivalent.
    \begin{enumerate}[label=\textup{(\arabic*)}]
        \item $\SM$ is explained by a labeled tree. 
        \item $\MDstrong(\SM)$ contains no prime modules $M$ with $|M|>1$. \label{thm:item-M1}
        \item $(\MDT_\SM,\tau_\SM)$ contains no \emph{\textsc{Prime}}-labeled inner vertices.
        \item There is no subset $W\subseteq X$ with $3\leq |W|$ for which $\SM[W]$ is primitive, i.e., 	there is no truly-primitive substructure.
        \item There is no subset $W\subseteq X$ with $3\leq |W|\leq 4$ for which $\SM[W]$ is primitive. 
        \item $\SM$ is $P_4$-free and rainbow-triangle-free.
    \end{enumerate}
\end{theorem}
\begin{proof}
	To see that (1) implies (2), let $(T,t)$ be a labeled tree and $\SM=(X,\Upsilon,\sigma)$ be the
	strudigram that is explained by $(T,t)$. Let $M\in \MDstrong(\SM)$ be an arbitrary strong module
	with $|M|>1$ and $v\coloneqq \lca_{T}(M)$ be the least common ancestor of the vertices in $M$.
	Since $|M|>1$, the vertex $v$ cannot be a leaf and has, therefore, at least two children. Let $u$ be
	a child of $v$ and put $M'\coloneqq L(T(u))\cap M$ and $M''\coloneqq M\setminus M'$. By
	construction, $\lca_{T}(x,y)=v$ for all $x\in M'$ and $y\in M''$. Since $(T,t)$ explains $\SM$
	and $\sigma(xy)=k$ for some $k\in \Upsilon$, we have $t(\lca_{T}(x,y))=k$ for all $x\in M'$ and
	$y\in M''$. Hence, $\SM[M] = \SM[M']\join_k\SM[M'']$. Consequently, $M$ is a $k$-series modules
	and, therefore, not prime. Clearly, (2)$\implies$(3)$\implies$(1) follows directly from the
	respective definitions. Thus, Conditions (1), (2) and (3) are equivalent. The equivalence between 
	(3), (4) and (5) is provided in \cite[Thm.~3.6]{Engelfriet1996} 
	where the statements were proven for the equivalent concept of
	labeled 2-structures. By Theorem\ \ref{thm:char-treeEx}, Conditions (1) and (6)
	are equivalent, which completes the proof. 
\end{proof}

The hierarchical structure
of $\MDstrong(\SM)$ implies that there is a unique partition $\Mmax(\SM) = \{M_1 , \dots , M_k\}$ of
$V(\SM)$ into the (inclusion-)maximal strong modules $M_j \neq V(\SM)$ of $\SM$
\cite{Moehring:84,Ehren1990A}. In particular, for any two disjoint modules $M,M'$ of $\SM$ there is
a $k\in \Upsilon(\SM)$ such that $\sigma(xy) = k$ for all $x\in M$ and $y\in M'$
\cite[L~4.11]{Ehren1990A}.  The latter two arguments allow us to define the \emph{quotient}
$\SM/\Mmax(\SM) = (\Mmax(\SM), \Upsilon, \sigma')$ where $\sigma'(MM')=\sigma(xy)$ for all $M,M'\in
\Mmax(\SM)$ and an arbitrary $x\in M$ and $y\in M'$.
Since $\Mmax(\SM)$ yields a partition of $V(\SM)$ and since 
$\sigma'(MM')=\sigma(xy)$ for all distinct $M,M'\in \Mmax(\SM)$ and all $x\in M$ and $y\in M'$
we obtain, 
\begin{observation}\label{obs:quotient} 
The quotient  $\SM/\Mmax(\SM)$ with $\Mmax(\SM) = \{M_1 , \dots , M_k\}$ is 
cr-isomorphic to the induced strudigram $\SM[W]$ for any 
$W\subseteq X$ that satisfies $|M_i \cap W | = 1$ for all $i \in \{1, \dots,k\}$.
\end{observation}

While there is a clear distinction between $\SM$ being $k$-series or prime, there are
strudigrams that are neither $k$-series nor primitive. By way of example, consider $\SM =
(\{y_1,y_2,x_1,x_2\},\Upsilon,\sigma)$ with $\sigma(x_iy_j)=j$ for $i,j\in \{1,2\}$ and
$\sigma(y_1y_2)\neq 1,2$. One easily verifies that $\SM$ cannot be written as a $k$-join and that
$\SM$ contains the non-trivial modules $\{x_1,x_2\}$. In addition, there are strudigrams that are
$k$-series and primitive e.g.\ any strudigram on two vertices. For quotients, however, this
situation changes. 
\begin{theorem}[{\cite[Thm.\ 2.17]{Engelfriet1996}}]\label{thm:Mprime=>QuotientMtruly-primitive}
Let $M$ be a prime module of $\SM$ with $|M|>1$. Then, $|M|\neq 2$ and $\SM[M]/\Mmax(\SM[M])$ is truly-primitive. 
\end{theorem}

\section{Prime-vertex Replacement}
\label{sec:pvr}

One of the main features of the modular decomposition tree $(\MDT,\tau)$ of a strudigram $\SM =
(X,\Upsilon,\sigma)$ is that, in the absence of prime modules, $(\MDT,\tau)$ explains $\SM$.
However, the information about $\sigma(x,y)$ is hidden if $\tau(\lca_{\MDT}(x,y)) = \primeL$. To
provide the missing structural information of $\SM$, we aim to locally replace the prime vertices in
$(\MDT,\tau)$ to obtain networks that capture full information about $\SM$. This idea has already
been fruitfully applied by replacing prime vertices with 
so-called $0/1$-labeled galls to obtain networks explaining restricted ``binary'' strudigrams
$(X,\{0,1\},\sigma)$ \cite{HS:22}, or with so-called labeled halfgrids to obtain networks that explain
arbitrary strudigrams \cite{BHS:21}. These constructions, involving the replacement of prime
vertices by other networks, have been studied for rather restricted types of networks only. Here, we
aim to generalize this concept to make the "prime-vertex replacement" applicable in a more general
setting.

There are several motivations to consider such networks. On the one hand, structural properties of
$(N,t)$ can provide deep structural insights into the strudigrams they explain. For example,
strudigrams $\SM = (X,\{0,1\},\sigma)$ that can be explained by so-called labeled galled trees (also
known as \gatex graphs) are closely related to weakly-chordal, perfect graphs with perfect order,
comparability, and permutation graphs \cite{HS:24}. These results are based on a characterization of
such \gatex graphs in terms of 25 forbidden induced subgraphs, a result established by leveraging
the structure of the network that explains $\SM$.

In addition, such networks $(N,t)$ can guide algorithms to solve many NP-hard problems in polynomial
time for restricted strudigrams. For example, in the absence of prime modules in $\SM =
(X,\{0,1\},\sigma)$, the graph representation $G$ of $\SM$ is a cograph. In this case, the modular
decomposition tree $(\MDT,\tau)$ has been shown to be a vital tool to prove that many NP-hard
problems such as finding optimal colorings, clique detection, Hamiltonicity, and cluster-deletion
can be solved in linear time on cographs \cite{BLS:99,CLS:81,Gao:13}. Similarly, if $(N,t)$ is a
so-called $0/1$-labeled galled tree, the resulting strudigram $\SM$ is equivalent to \gatex graphs, and it has
been shown in \cite{HS:24B,hellmuth2024solving} that NP-hard problems such as finding 
perfect orderings, maximum cliques or independent sets, and minimum colorings can be solved in
linear time. Again, $(N,t)$ serves as a guide for these algorithms to find the optimal solutions.

Modular decomposition has many important applications, e.g., in community detection in complex
networks \cite{jia2015defining}, in evolutionary biology and orthology detection
\cite{Hellmuth:13a,Hellmuth:15a,lafond2015orthology}, in structure recognition of
protein-interaction networks \cite{DER:09,gagneur2004modular}, pattern detection in diagnostic data
\cite{piceno2021co}, graph drawing \cite{papadopoulos2006drawing}, automata theory
\cite{biggar2021modular}, and many more. However, graphs or, more generally, strudigrams that arise
from real-world data, particularly in evolutionary biology, are often better represented by a
network rather than a tree to accommodate evolutionary events such as hybridization or horizontal
gene transfer. One way to address this problem is, again, the local replacement of prime
vertices in $(\MDT,\tau)$ by small labeled networks.  

We start here by providing the general concept of  prime-vertex replacement (pvr) networks.
Later, in Section \ref{sec:gatex}, we will focus on a particular type of DAGs used to replace prime vertices
and provide a characterization of strudigrams that can be explained by such types of networks.

In what follows, we say that two networks $N$ and $N'$ are internal vertex-disjoint, if $(V^0(N)\setminus \{\rho_N\}) \cap (V^0(N')\setminus \{\rho_{N'}\}) = \emptyset$.

\begin{definition}[\bf prime-explaining family]\label{def:pfam}
	Let $\SM = (X,\Upsilon,\sigma)$ be a strudigram and $\P\subseteq \MDstrong(\SM)$ be the set of all of its prime modules and 
	$(\MDT,\tau)$ be its modular decomposition tree. A \emph{prime-explaining family (of $\SM$)} is a set
	$\pfam(\SM)=\{(N_M,t_M) \mid M\in\P\}$ of pairwise internal vertex-disjoint  labeled networks such
	that, for each $M\in\P$, the network $(N_M,t_M)$ explains the strudigram $\SM[M]/\Mmax(\SM[M])$. 
\end{definition}

As we shall see, the modular decomposition tree $(\MDT,\tau)$ of each strudigram $\SM$ can be 
modified by locally replacing prime vertices $M\in  \P$ by $(N_M,t_M)\in \pfam(\SM)$ to obtain 
a labeled network that explains $\SM$. This idea is made precise in the following
\begin{definition}[\bf prime-vertex replacement (pvr) networks]
  \label{def:pvr}
	Let $\SM=(X,\Upsilon,\sigma)$ be a strudigram and $\P$ be the set of all prime vertices in
	its modular decomposition tree $(\MDT,\tau)$. 
	Let $\pfam(\SM)$ be a prime-explaining family of $\SM$.
	A \emph{prime-vertex replacement (pvr) network} $\pvr(\SM,\pfam(\SM))$ of $\SM$ is the directed, labeled graph $(N,t)$ obtained by the following procedure: \smallskip
\begin{enumerate}
\item For each $M\in \P$, remove all edges $(M,M')$ with
  		$M'\in \child_{\MDT}(M)$ from $\MDT$ to obtain the directed graph
  		$T'$.\label{step:T'}\smallskip
\item Construct a directed graph $N$ by adding, for each $M\in\P$,  $N_M$ to $T'$ by identifying the root
  		of $N_M$ with $M$ in $T'$ and each leaf $M'$ of $N_M$ with the
	  	corresponding child $M'\in \child_{\MDT}(M)$. \label{step:T''}\smallskip
\item \label{step:color} 
 Define the labeling $t\colon V^0(N)\to \Upsilon$ by putting, for
  all $w\in V^0(N)$,
  \begin{equation*}
    t(w) \coloneqq
    \begin{cases} 
		 t_{M}(w) &\mbox{if }  w \in V^0(N_{M}) \text{ for some } M\in \P\\
      \tau(w) &\mbox{otherwise, i.e., if } w\in V^0(\MDT)\setminus \P 
    \end{cases}
  \end{equation*}  
\item Replace every singleton module $\{x\}\in V(N)$ in $N$ by $x$.\label{step:flatten leaves}\smallskip
\end{enumerate}
The resulting labeled directed graph  $(N,t)$ is called \emph{pvr-network (of $\SM$)} and denoted by $\pvr(\SM,\pfam(\SM))$.
\end{definition}

Step \ref{step:flatten leaves} of Definition~\ref{def:pvr} is a necessary
technicality brought by the fact that the leaves of the modular decomposition tree consists of the
singletons $\{x\}$ for $x\in V(\SM)$, rather than the set $V(\SM)$ itself. 
We emphasize that the networks used in $\pfam(\SM)$ are not required to be leaf-separated.
See Figure \ref{fig:exmapl1-pvr} and \ref{fig:gatex_pvr_8} for illustrative examples.

\begin{lemma}\label{lem:pvr-well-defined}
	Let $\SM=(X,\Upsilon,\sigma)$ be a strudigram, $\P$ be the set of all prime vertices in
	its modular decomposition tree $(\MDT,\tau)$ and  
	$\pfam(\SM)$ be a prime-explaining family of $\SM$.
	Then,  $\pvr(\SM,\pfam(\SM))$ is a labeled network on $X$.
	
	In particular,  $V^0(\MDT)\cup X \subseteq V(N)$. Moreover, all $u,v\in V(\MDT)$ with $v\preceq_\MDT u$ satisfy $v\preceq_N u$, and 
	every vertex along the (unique) directed $uv$-path in $\MDT$ is contained in every directed $uv$-path in $N$. 	
\end{lemma}
\begin{proof}
	Let $(N,t) \coloneqq \pvr(\SM,\pfam(\SM))$. We first verify that $N$ is a well-defined network on
	$X$. By construction, we have $V(N)=V^0(\MDT)\cup X\cup\left(\bigcup_{M\in\P}V^0(N_M)\right)$. Note that
	each $M\in\P$ is a vertex in $V^0(\MDT)$ and thus, the identification of the root $\rho_{N_M}$ of
	$N_M$ with the vertex $M$ in $\MDT$ in Def.~\ref{def:pvr}\eqref{step:T''} is well-defined.
	Furthermore, by definition of MDTs, $\child_\MDT(M)=\Mmax(\SM[M])$ for all strong modules $M$ of
	$\SM$. Moreover, the quotient strudigram $\SM_M\coloneqq\SM[M]/\Mmax(\SM[M])$ has vertex set
	$\Mmax(\SM[M])$ and, by definition of $\pfam(\SM)$, the rooted network $(N_M,t_M)\in \pfam(\SM)$
	explains $\SM_M$ and must, therefore, satisfy $L(N_M)=V(\SM_M)=\Mmax(\SM[M])=\child_\MDT(M)$.
	Hence there is a 1-to-1 correspondence between the leaves of $N_M$ and the elements of
	$\child_\MDT(M)$, and the identification of vertices in Def.~\ref{def:pvr}\eqref{step:T''} is
	well-defined. Moreover, the only edges removed from $\MDT$ while constructing $N$ are edges $(M
	,M')$ for which $M\in\P$ and $M'\in\child_\MDT(M)$, and every edge of $N_M$ with $M\in\P$ appears
	in $N$. Since $N_M$ is a network on $\Mmax(\SM[M])$, there is a directed path from $\rho_{N_M}$ to each
	element in  $\Mmax(\SM[M]) = \child_\MDT(M)$ and these paths remain in $N$. Hence,
	any $v,u\in V^0(N_M)$ with $v\preceq_{N_M} u$ satisfy $v\preceq_N u$. It is now a straightforward
	task to verify that the latter argument together with the construction of $N$ immediately implies
	that, for any $v,u\in V(\MDT)$ with $v\preceq_\MDT u$, there is a directed path $uv$-path in $N$, 
	no directed path $vu$-path in $N$ and every vertex along the (unique)
	directed $uv$-path in $\MDT$ is contained in every directed $uv$-path in $N$. Consequently, $N$
	is a well-defined DAG. Moreover, it is an easy task to verify that the module $X$ remains, by
	construction, the unique root of $N$. In addition, $\{x\}\notin \P$ for all $x\in X$ and by
	Def.~\ref{def:pvr}\eqref{step:flatten leaves} each $x\in X$ is a vertex in $N$. Taken the latter
	two arguments together, $N$ is a well-defined network on $X$. 
	
	By construction all vertices in $V^0(\MDT)$ remain vertices in $N$ and thus, $V^0(\MDT)\cup X \subseteq V(N)$. 
	Note that each vertex $w\in V^0(N)$ is either a vertex in $V^0(\MDT)$ or a vertex contained in $V^0(N_M)$ for some $M\in \P$
	and each such vertex obtained either some color $\tau(w)$ or $t_M(w)$ according to Def.~\ref{def:pvr}\eqref{step:color}. 
	Consequently, 
	the labeling $t$ as specified in Def.~\ref{def:pvr}\eqref{step:color} is well-defined. 
	Hence, $(N,t)  = \pvr(\SM,\pfam(\SM))$ is a labeled network on $X$.
\end{proof}

\begin{figure}
	\centering
	\includegraphics[width=.8\textwidth]{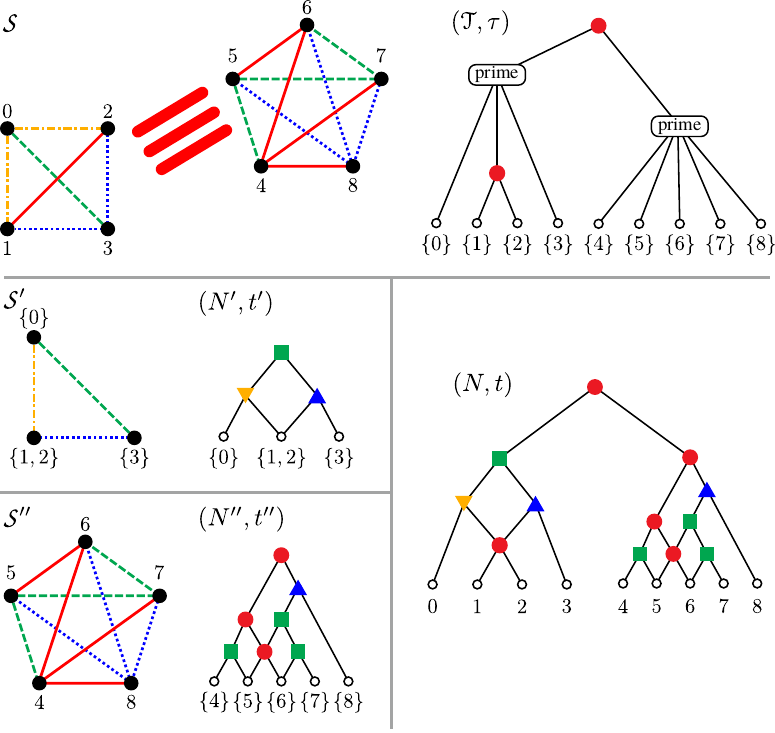}
	\caption{	An example to illustrate the idea of pvr-networks, see Fig.~\ref{fig:exmpl-sketch1} for the color legend.
				   Shown is the edge-colored graph representation of a strudigram $\SM$ on $X  = \{0,1,\dots,8\}$ (in which $\sigma(xy)=$``red'' for all $0\leq x\leq 3$ and $4\leq y\leq 8$, indicated by three red thick edges)
				   together with its MDT $(\MDT, \tau)$ (right to $\SM$). The strudigram
				   $\SM$ contains precisely three non-trivial strong modules, namely $M_0 =\{1,2\}$, $M_1 =\{0,1,2,3\}$ and $M_2 =\{4,5,6,7,8\}$.
					Hence, $\MDstrong(\SM) = \{ \{0\},\{1\},\ldots,\{8\},M_0,M_1,M_2,X\}$. 
				   Both modules $M_1$ and $M_2$ are prime and we have $\P = \{M_1,M_2\}$.
				   Here $\SM[M_1]$ is prime but not primitive since $\{1,2\}\in \MDstrong(\SM)[M_1]$.  
				   The quotient $\SM' = \SM[M_1] / \Mmax(\SM[M_1])$ is primitive, where  $\Mmax(\SM[M_1]) = \{\{0\},\{3\},\{1,2\}\}$.
					The network $(N',t')$ explains $\SM'$. 
				   Moreover, $\SM'' = \SM[M_2] = \SM[M_2] / \Mmax(\SM[M_2])$ is primitive and 
				   is explained by the network $(N'',t'')$.
					In this example, we thus obtain the prime-explaining family
				   $\pfam(\SM)=\{(N',t'), (N'',t'')\}$ of $\SM$. 				   
				    Replacing the prime vertices in  $(\MDT, \tau)$ 
				   by the respective network as specified in Def.\ \ref{def:pvr} yields the network $(N,t)$ that explains $\SM$. }
	\label{fig:exmapl1-pvr}
\end{figure}

The important property of a pvr-network is that it can be used to explain strudigrams, ensured by the following two results. 

\begin{lemma}\label{lem:pvr pairwise lca}
	Let $\SM=(X,\Upsilon,\sigma)$ be a strudigram, $\P$ be the set of all prime vertices in
	its modular decomposition tree $(\MDT,\tau)$ and  	$\pfam(\SM)$ be a prime-explaining family of $\SM$.
	Put $(N,t)\coloneqq \pvr(\SM,\pfam)$ and $\dot v\coloneqq \{v\}$.
	Then, $N$ has the  $\2lca$-property and, for all distinct $x,y\in L(N)$, it holds that
	\[\lca_N(x,y)=\begin{cases}
		\lca_{\MDT}(\dot x,\dot y)&\text{if }\lca_{\MDT}(\dot x,\dot y)\notin\P\\
			\lca_{N_M}(v_x,v_y)&\text{if }\lca_{\MDT}(\dot x,\dot y)=M\in\P,
			\text{where $v_x,v_y\in L(N_M)$ are the unique} \\ 
			&\text{vertices identified with $M_x$ and $M_y$ in $\MDT$, respectively and} \\ 
			&\text{where $x\in M_x$ and $y\in M_y$.}
	\end{cases}\]
		Moreover, if every $(N_M,t_M)\in \pfam(\SM)$ has the $\klca$-property for some $k\leq |X|$,
		then $(N,t)$ has the $\klca$-property.
\end{lemma}
\begin{proof}
	
	Let $\SM=(X,\Upsilon,\sigma)$ be a strudigram, $\P$ be the set of all prime vertices in
	its modular decomposition tree $(\MDT,\tau)$ and $\pfam(\SM)$ be a prime-explaining family of $\SM$.
	Put $(N,t)\coloneqq \pvr(\SM,\pfam)$. 
	Note that every $(N_M,t_M)\in \pfam(\SM)$ has the $\2lca$-property. 
	Assume now that that every $(N_M,t_M)\in \pfam(\SM)$  has the $\klca$-property for some $k\geq 2$. 
	Let $A\subseteq X$ with $|A|\leq k$. 
	Since $\MDT$ is a tree, it is an lca-network and, therefore, $M\coloneqq \lca_{\MDT}(A)$ 
	is well-defined. As already argued in the proof of Lemma~\ref{lem:pvr-well-defined}, $L(N_M)=\Mmax(\SM[M])=\child_\MDT(M)$, for all $M\in\P$.
   Let $M_1,\dots,M_\ell \in \child_\MDT(M)$ be all children of $M$ for which $A\cap M_i\neq \emptyset$. 
	Since $M_1,\dots, M_\ell \in \Mmax(\SM[M])$, we have $M_i\cap M_j = \emptyset$, $1\leq i<j\leq \ell$.
	In particular, $\ell\leq k$ must hold. Moreover, 
	$M$ is the unique inclusion-minimal cluster in $\MDT$ that contains $B \coloneqq \cup_{i=1}^\ell M_i$. The latter arguments
	together with Obs.\ \ref{obs:deflcaY} imply
	$M = \lca_{\MDT}(A) = \lca_{\MDT}(B)$.	
	
	By Lemma \ref{lem:pvr-well-defined}, $N$ is a network on $X$ and thus, 
	the root $\rho_N$ is an ancestor of all vertices in $N$. 	By construction, none of the 
	networks $N_{M'}$ used to replace a prime module $M'\neq M$ along the 
	path from $M$ to any vertex $\{x\}$, in $x\in M_i$, in $\MDT$ shares vertices or edges with any 
	network $N_{M''}$ used to replace a prime module $M''	\neq M$ along the 
	path from from $M$ to $\{y\}$, $y\in M_j$, in $\MDT$, $1\leq i<j\leq \ell$. 
	Let us denote, for a strong module $\tilde M$, with $U(\tilde M)$ either the set $\{ \tilde M\}$ in case $\tilde M\notin \P$
	or the set of vertices in $V^0(N_{\tilde M})\cup \{\tilde M\}$ in the network $N_{\tilde M}$ used to replace $\tilde M$ in case $\tilde M\in \P$.
	By construction of $N$ and Lemma \ref{lem:pvr-well-defined}, 
	for any strong module $M'''$ of $\SM$ with $M\prec_{\MDT} M'''$ we 
	have $v\prec_{\MDT} u$ for all $v\in U(M)$ and   $u\in U(M''')$. 
	Moreover, if $M$ and $M'''$ are $\prec_{\MDT}$-incomparable in $\MDT$, 
	then the vertices in $U(M)$ are pairwise $\prec_{N}$-incomparable with the vertices in  $U(M''')$ . 
	In other words, no vertex in $U(M''')$	can be a least common ancestors of $A$ for all such $M'''$. 
	It now follows that either 	$M = \lca_{N}(A) = \lca_{N}(B)$ in case
	$M\notin \P$ or that the least common ancestors of $A$ and $B$ are located in $V(N_M)$ in case	$M\in \P$. 
	In the latter case, the least common ancestors of $A$ and $B$ coincide and, they are equal
	to the least common ancestors of $\{M_1,\dots,M_\ell\} \subseteq L(N_M)$. Since 
	 $\ell\leq k$ and $N_M$ has the $\klca$-property, 
	 $\lca_{N_M}(\{M_1,\dots,M_\ell\}) = \lca_N(A)$ is well-defined.
	 Since $A\subseteq X$ with $|A|\leq k$ were arbitrarily chosen, $N$ has the $\klca$-property. 
	 
	In particular, since every $(N_M,t_M)\in \pfam(\SM)$ has the $\2lca$-property, 
	$N$ must have the $\2lca$-property.
	Reusing the latter arguments, it is a straightforward task to verify that 
	$\lca_N(x,y)=\lca_{\MDT}(\{x\},\{y\})$ in case $M = \lca_{\MDT}(\{x\},\{y\})\notin\P$
	and, otherwise, $\lca_N(x,y)  = \lca_{N_M}(v_x,v_y)$	where $v_x,v_y\in L(N_M)$ are the unique vertices identified with
	$M_x$ and $M_y$ in $\MDT$, respectively.
\end{proof}

\begin{proposition}\label{prop:pvr-explains-S}
	The pvr-network $\pvr(\SM,\pfam(\SM))$ explains a strudigram $\SM$ for  all  prime-explaining families $\pfam(\SM)$.
\end{proposition}
\begin{proof}
	Let $\SM=(X,\Upsilon,\sigma)$ be a strudigram, $(\MDT,\tau)$ its modular decomposition tree, 
	 and put $(N,t)\coloneqq \pvr(\SM,\pfam(\SM))$. 
	To simplify writing, we write $\dot x$ for the the singleton modules $\{x\} \in L(\MDT)$. 
	Let $x,y\in X$ be distinct vertices. 
	Since $\MDT$ is a tree, $M\coloneqq\lca_\MDT(\dot x,\dot y)$ 
	is well-defined. Aided by Lemma~\ref{lem:pvr pairwise lca}, we consider two cases: 
	$M\notin\P$, respectively $M\in\P$.

	Assume that $M\notin\P$, so that $M=\lca_\MDT(x,y)\in V(\MDT)\setminus\P$. Applying 
	Lemma~\ref{lem:pvr pairwise lca} and $t$ as specified in Def.~\ref{def:pvr}\eqref{step:color} 
	ensures  $t(\lca_{N}(x,y))=\tau(M)$. On the other hand, the definition of MDTs ensures 
	that $\sigma(xy)=\tau(M)$, since $\tau(v)\neq\primeL$. 
	Thus $t(\lca_{N}(x,y))=\sigma(xy)$ holds.

	Assume now that $M\in\P$. Since $\pfam(\SM)$ is a prime-explaining family, the network $(N_M,t_M)$ explain the strudigram
	$\SM_M\coloneqq\SM[M]/\Mmax(\SM[M])$, where $M=L(\MDT(M))$. In other words we have that 
	$L(N_M)=\Mmax(\SM[M])$ and $\sigma_M(M'M'')=t_M(\lca_{N_M}(M',M''))$
	for all distinct $M',M''\in\Mmax(\SM[M])$, where $\sigma_M$ denotes the map of $\SM_M$.
	Due to Observation~\ref{obs:quotient}, this means that if $M',M''\in\Mmax(\SM)$ are distinct, 
	and if $u\in M'$ respectively $v\in M''$ are vertices of $\SM$, then 
	\begin{equation}
		\sigma(uv)=t_M(\lca_{N_M}(M',M'')).
		\label{eq:sigma equal to}
	\end{equation}
	Moreover, by construction of $N$ we have that $L(N_M)=\child_\MDT(M)=\Mmax(\SM[M])$, a set which, 
	since $M=\lca_\MDT(\dot x,\dot y)$, must contain two distinct  $M_x$ and $M_y$ which satisfy 
	$\dot x\preceq_\MDT M_x$ respectively $\dot y\preceq_\MDT M_y$. In particular, $M_x \cap M_y = \emptyset$. 
	By construction of $N$ and application of Lemma \ref{lem:pvr-well-defined}, 
	$M_x=L(N(M_x))$ and $M_y=L(N(M_y))$. 
	Since $x\in M_x$ and $y\in M_y$, Eq.~\eqref{eq:sigma equal to} and 
	Lemma~\ref{lem:pvr pairwise lca} imply that
	\begin{equation*}
		t(\lca_{N}(x,y))=t_M(\lca_{N_M}(M_x,M_y))=\sigma(xy).
	\end{equation*}
	In summary, $t(\lca_{N}(x,y))=\sigma(xy)$ for all $x,y\in X$ and, therefore, $(N,t)$ explains $\SM$.
\end{proof}

Lemma \ref{lem:pvr pairwise lca} together with Prop.\ \ref{prop:SM-hasExpl} and \ref{prop:pvr-explains-S} implies
\begin{corollary}\label{cor:pvr-lca-network}
	If every  $(N_M,t_M)\in \pfam(\SM)$ in the prime-explaining family of $\SM$
	has the $\klca$-property, then $\pvr(\SM,\pfam)$ is a network with $\klca$-property that explains $\SM$.

	In particular, for every strudigram $\SM$, there is a prime-explaining family $\pfam(\SM)$
	consisting solely of labeled lca-networks, in which case
	 $\pvr(\SM,\pfam(\SM))$ is an lca-network that explains $\SM$.
\end{corollary}


\section{Galled-trees and \gatex Strudigrams}
\label{sec:gatex}

As shown in the previous sections, every strudigram $\SM$ can be explained by some 
labeled lca-network which has the nice property that $\lca_N(A)$ is well-defined for all $A\subseteq L(N)$. 
However, these networks can become quite arbitrary, as illustrated in Figure \ref{fig:lca-N}.
We thus consider here much simpler lca-networks, so-called galled trees.
	
A  \emph{gall} $C$ is a subgraph of a directed graph $G$ such that $C$ is composed of two
internal-vertex disjoint di-paths that intersect only in their start- and end-vertex. To be more
precise, $C$ is a gall, if there are two directed $xy$-paths $P^1$ and $P^2$ such that $V(P^1)\cap
V(P^2)=\{x,y\}$ and such that $V(P^1)\cup V(P^2) = V(C)$ and $E(P^1)\cup E(P^2) = E(C)$. In this
case, $\rho_C\coloneqq x$ is the \emph{root} and $\eta_C\coloneqq y$ the \emph{hybrid} of $C$. The
paths $P^1$ and $P^2$ from which $C$ is composed are also called \emph{sides} of the gall $C$.
A \emph{galled-tree} is a phylogenetic network $N$ where every nontrivial biconnected component is a gall.
In this section, we characterize strudigrams that can be explained by labeled galled-trees. 
Note that in \cite{HS:22}, galled-trees have been called level-1 networks. 
Clearly, every \emph{phylogenetic} tree is a galled-tree. However, in general, galled-trees are not required to  be leaf-separated.

A particular type of tree of importance here is a
\emph{caterpillar tree}, whose inner vertices consist of a single (directed) path $P^1=v_1v_2\ldots
v_k$ (possibly $k=1$) such that $v_i$ has a single leaf-child for each $i=1,2,\ldots,k-1$ and such
that $v_k$ has precisely two children, both of which are leafs. Additionally, a caterpillar tree will
by definition be rooted at the vertex $v_1$. Any graph isomorphic to the subgraph of a caterpillar
tree induced by $v_k$ and its two children $x$ and $y$ is called a \emph{cherry}. In that case, the
leaves $x$ and $y$ are said to be \emph{part of a cherry}.

\begin{theorem}[{\cite[Thm.\ 11]{Hellmuth2023}}]
  \label{thm:galled-no3overlap}
  $\mathfrak{C}$ is the clustering system of a galled-tree if and only if
  $\mathfrak{C}$ is closed and satisfies (L) and (N3O).  
  In particular, $\Hasse[\mathfrak{C}]$ is a galled-tree.
\end{theorem}
\begin{proof}
	These statements have been proven in \cite{Hellmuth2023} for networks
	whose non-trivial biconnected components are galls but that are not
	necessarily phylogenetic. Clearly, the \emph{only-if} direction 
	remains true for galled-trees as defined here.	
	Suppose now that  $\mathfrak{C}$ is closed and satisfies (L) and (N3O).  
	By \cite[Thm.\ 11]{Hellmuth2023}, 
	$\Hasse[\mathfrak{C}]$ is a galled-tree, which,  in particular, 
	establishes the \emph{if} direction of the first statement.
\end{proof}

Recall that in lca-networks, the lca of any subset of leaves is well-defined. For galled-trees, the following stronger statement holds.
\begin{lemma}
	\label{lem:lev1-klca}
	Every galled-tree is a global lca-network.
\end{lemma}
\begin{proof}
	Let $N=(V,E)$ be a galled-tree. We may construct a network $N^*=(V^*,E^*)$ from $N$ by adding, for every
	non-leaf vertex $v\in V$ a new vertex $x_v$ to $V^*$ and the edge $(v, x_v)$ to $E^*$ (c.f.
	\cite[Def.~3.9]{SCHS:24}). Clearly $N^*$ remains a galled-tree. Lemma~49 of \cite{Hellmuth2023} implies that $N^*$  is an lca-network. 
	This allows us to apply Proposition~3.11 of \cite{SCHS:24} to conclude that $\lca(W)$ is well-defined for every
	$\emptyset\neq W\subseteq V(N)$. 
\end{proof}

From here on, we assume that the fact that a galled-tree
is a global lca-network is known and we avoid to mention it each time. In particular, every galled-tree
has the $\2lca$-property, and $\mathcal{S}(N,t)$ is thus defined for every labeled galled-tree
$(N,t)$. A strudigram the can be explained by a galled-tree is \gatex (Galled-Tree Explainable).

\begin{definition}\label{def:elementary}
	A galled-tree $N$ on $X$ is \emph{elementary} if it contains exactly one gall $C$ and satisfies the following conditions: 
	\begin{enumerate}[label=(\roman*)]
		\item Every vertex of $N$ is a vertex of $C$ or the child of a vertex in $C$-
		\item The root $\rho_C$ has no leaf-children, except possibly $\eta_C$.
		\item Each inner vertex $v\neq \rho_C$ has precisely one leaf-child $x$ with $\indeg_N(x)=1$.
	\end{enumerate}
\end{definition}

Note that by condition (i), $\rho_N=\rho_C$. By assumption, the gall $C$ is the
only gall of $N$ and thus every elementary galled-tree has a unique hybrid vertex, namely, $\eta_C$.
Moreover, this hybrid-vertex is either a leaf or has precisely one child and this child is a leaf.
See Figure~\ref{fig:elementary} and its caption for some descriptive examples.

\begin{figure}
\includegraphics[width=\textwidth]{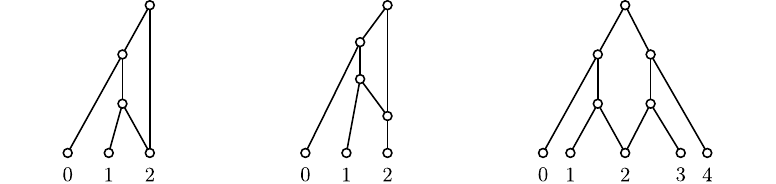}
\caption{Three elementary galled-trees are shown. Note, in particular, that an elementary galled-tree can be leaf-separated (middle) or not (left resp. right). Only the right-most network is a strong galled-tree.
}\label{fig:elementary}
\end{figure}

A labeled galled-tree $(N,t)$ is \emph{discriminating} if, for all adjacent inner vertices $u$ and $v$, we
have $t(u)\neq t(v)$. If $t(u)\neq t(v)$ for all edges $(u,v) \in E(N)$ where $v$ is an inner tree-vertex, then
$(N,t)$ is \emph{quasi-discriminating}. 
Since trees have no hybrid
vertices any quasi-discriminating tree $(T,t)$ is  discriminating.

\begin{definition}
For a path $P$, let $\widetilde V(P)$ be the set of vertices in $P$ that are distinct from its end vertices.
A gall $C$ with sides $P_1$ and $P_2$ in a labeled galled-tree $(N,t)$ is \emph{strong} if 
it satisfies \smallskip
\begin{enumerate}[label=(S\arabic*)]
\item $\widetilde V(P^i) \neq \emptyset$ for $i\in \{1,2\}$,  and \label{def:strong1}\smallskip
\item if $|\widetilde V(P^1)|= |\widetilde V(P^2)|=1$, then $|\{t(v_1), t(v_2),t(\rho_C)\}|= 3$ 
where $\widetilde V(P^i) = \{v_i\}$  for $i\in \{1,2\}$. \label{def:strong2} \smallskip
\end{enumerate}
A galled-tree is \emph{strong} if all its galls are
\emph{strong}.
\end{definition}
Note that trees vacuously are strong galled-trees. See also Figure~\ref{fig:elementary} and Figure~\ref{fig:rainbowT} for further (non-)examples of strong (elementary) galled-trees.

In the following, we will focus on labeled galled-trees that are strong, elementary, and quasi-discriminating.
As we will soon see, this type of galled-tree is closely connected to the following class of strudigrams.

\begin{definition}\label{def:polarcat}
    We say that a strudigram $\SM=(X,\Upsilon,\sigma)$ is a \emph{polar-cat} if there exists a vertex $x\in X$ and 
    induced strudigrams $\SM_1$ and $\SM_2$ of $\SM$ such that \smallskip
    \begin{enumerate}
        \item[(i)] $\SM-x = (\SM_1-x)\join_k (\SM_2-x)$ for some $k\in \Upsilon$.
        \item[(ii)] For $i\in\{1,2\}$, $\SM_i$ can be explained by a labeled discriminating caterpillar tree $(T_i,t_i)$ such that 
        \begin{enumerate}
            \item[(a)] $x$ is part of the cherry in $T_i$, and 
            \item[(b)] $t_i(\rho_{T_i})\neq k$.
        \end{enumerate} \smallskip
    \end{enumerate}
    In that case, we also say that $\SM$ is a \emph{$\join_k$-$(x, \SM_1, \SM_2)$-polar-cat}. 
\end{definition}
For an illustrative example of Definition~\ref{def:polarcat}, see Fig.~\ref{fig:defNpc}.
The concept of polar-cats was introduced in \cite{HS:22} in the context of graphs (without
edge-colors) that are explained by galled-trees, and Def.~\ref{def:polarcat} is a direct generalization
to the case of strudigrams. The name polar-cat stems from the fact that the respective substrudigrams $\SM_1$ and $\SM_2$ 	are \textbf{polar}izing,
i.e. the requirement of $t_i(\rho_{T_i})\neq k$, and \textbf{cat}erpillar-explainable.
We note in passing the polar-cats as defined in  \cite{HS:22} must have at least four vertices. 
We relaxed this condition to deal with the general case of rainbow-triangles.

In fact, rainbow-triangles are worth remarking upon further. By Theorem~\ref{thm:unp<=>treeExpl}\ref{thm:item-M1}, any strudigram on two vertices can be
explained by a labeled tree. Moreover, it is an easy task to verify that among all strudigrams on
three vertices, only the rainbow-triangle is primitive, and thus, cannot be explained by a tree,
see Fig.\ \ref{fig:rainbowT}. The latter arguments together with Theorem~\ref{thm:unp<=>treeExpl}
imply	

\begin{observation}\label{obs:rainbowtriangle}
	The smallest strudigram that is not explained by a labeled tree is the rainbow-triangle
	and rainbow-triangles are the only strudigrams on three vertices that are primitive. 
	In particular, rainbow-triangles are truly-primitive.
	Nevertheless, rainbow-triangles are \textup{\gatex} and can, in particular, be explained by a 
	strong, elementary, quasi-discriminating galled-tree.
\end{observation}

	\begin{figure}
		\centering
		\includegraphics[width=0.9\textwidth]{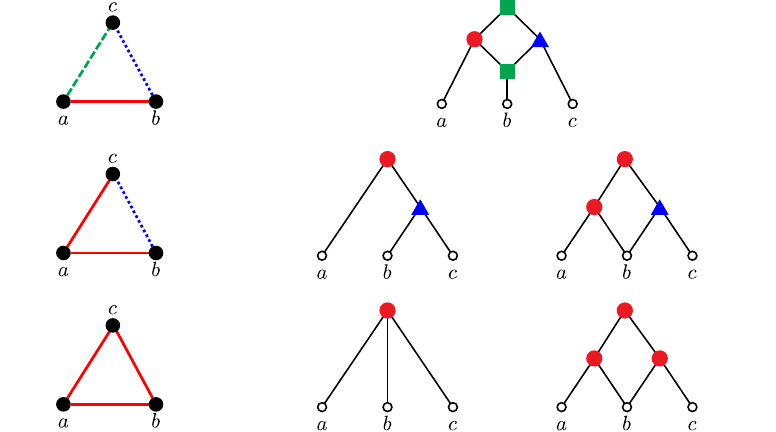}
		\caption{Left, all strudigrams on three vertices (up to choice and permutation of colors) are shown, with color legend as in Fig.~\ref{fig:exmpl-sketch1}. 
					Galled-trees that explain $\SM$ are drawn in the same row as each individual strudigram $\SM$.		Only the rainbow-triangle $\TM$ (top-left)
					cannot be explained by a labeled tree, since it is primitive. Nevertheless it is explained by 
					a strong, elementary and quasi-discriminating galled-tree. In particular, the depicted galled-tree is precisely $(\N(b,\TM_1,\TM_2),t(b,\TM_1,\TM_2))$ for $\TM_1=\TM[\{a,b\}]$ resp. $\TM_2=\TM[\{b,c\}]$, as specified in Def.~\ref{def:N(pc)}.
					Each of the other two labeled galled-trees in the middle and bottom row that contain a gall are elementary but neither strong nor quasi-discriminating.
		}
		\label{fig:rainbowT}
		\end{figure}

\begin{lemma}\label{lem:morerainbowtriangle}
Every polar-cat has at least three vertices. 
Moreover, any rainbow-triangle $\SM = (\{x_1,x_2,x_3\}, \Upsilon, \sigma)$
is a  $(x_i,\SM[\{x_i,x_j\}],\SM[\{x_i,x_l\}])$-polar-cat, $\{i,j,l\} = \{1,2,3\}$. 
\end{lemma}
\begin{proof}
	Let $\SM=(X,\Upsilon,\sigma)$ be a $\join_k$-$(x, \SM_1, \SM_2)$-polar-cat. 
	By Def.~\ref{def:polarcat}(ii), each of $\SM_1$ and $\SM_2$ is explained by a labeled caterpillar tree 
	such that $x$ is part of the cherry in each of the respective  trees. Hence, each of $\SM_1$ and $\SM_2$ contains
	at least two vertices and both contain $x$. This together with Def.~\ref{def:polarcat}(i)
	implies that $\SM$ has at least three vertices. 
	A generic example, that shows that a rainbow-triangle $\SM = (\{x_1,x_2,x_3\}, \Upsilon, \sigma)$ 
	is a  $(x_i,\SM[\{x_i,x_j\}],\SM[\{x_i,x_l\}])$-polar-cat, $\{i,j,l\} = \{1,2,3\}$ is provided
	in Fig.\ \ref{fig:rainbowT}(left) where we can put $\{x_1,x_2,x_3\} = \{a,b,c\}$. 
\end{proof}

It is easy to see that strudigrams $\SM=(X,\Upsilon,\sigma)$ with either $|X|\leq 2$ or
$|X|=3$ and where the three pairs of vertices in $X$ receive the same color in $\Upsilon$, can never satisfy Def.\ \ref{def:polarcat}(ii)
and are, therefore, not polar-cats. Nevertheless, if  $|X|=3$ and we have 2 colors in $\Upsilon$ distributed among
the three pairs of vertices in $X$ we obtain a polar-cat: simply choose $x\in X$ as the vertex that satisfies
$\sigma(xv)=\sigma(xu)$ for $u,v\in X\setminus\{x\}$. 
Thus, all other polar-cats are either rainbow-triangles or have at least four vertices and we call
them \emph{handsome}.

\begin{lemma}\label{lem:seqdGaT=>pc}
	If $\SM$ is a strudigram that is explained by a strong, elementary and quasi-discriminating galled-tree, 
	then $\SM$ is a handsome polar-cat. 
\end{lemma}
\begin{proof}
    Let $(N,t)$ be a strong, elementary and quasi-discriminating galled-tree and let
    $\SM=(X,\Upsilon, \sigma)$ be the strudigram that is explained by $(N,t)$.  Since $N$ is
    elementary, it contains a unique gall $C$. Let $P^1$ and $P^2$ denote the sides of $C$. Recall
    $P^1$ and $P^2$ meet only in the root $\rho$ and the (unique) hybrid vertex $\eta$ of $N$. Let
    $W_i$ be the set of the leaf-children of the vertices in $P^i$ and put $\SM_i=\SM[W_i]$, for
    $i\in\{1,2\}$. Furthermore, to ease notation, we let $v$ denote the child of $\eta$ if such
    vertex exist; otherwise $v=\eta$. Thus, $v$ is always the leaf of $N$ that satisfies $v\preceq_N
    \eta$. Clearly, $W_1\cap W_2=\{v\}$. Finally, put $k\coloneqq t(\rho)$.
    
    It is an easy task to verify that $\lca(x,y)=\rho_N$ for all $x\in   W_1\setminus\{v\}$ and $y\in W_2\setminus\{v\}$. 
    Thus, $t(\lca(x,y))=t(\rho)=k$ for all $x\in  W_1\setminus\{v\}$ and $y\in W_2\setminus\{v\}$. 
    Since $\SM$ is explained by $(N,t)$, it holds that $\sigma(xy)=k$ for all $x\in  W_1\setminus\{v\}$ and $y\in W_2\setminus\{v\}$. 
    Hence, $\SM-v = (\SM_1-v)\join_{k} (\SM_2-v)$ and $\SM_1$ and $\SM_2$ 
    satisfy  Definition~\ref{def:polarcat}(i).  
   
    We now construct a caterpillar tree that explains $\SM_1$. Suppose $u$ is the child of $\rho$ that lies on $P^1$, and let $(T_1, t_1)$ be
	be the labeled subnetwork where $T_1 = N(u)$ and where $t_1$ is the restriction of $t$ to $T_1$. Note that $u \neq \eta$ since $N$ is strong.
	Since $N$ is phylogenetic, elementary and strong, 
	only $\eta$ may have been suppressed while constructing $N(u)$ (which is the case if $\eta$ is not already a leaf in $N$).
	One easily verifies that  $T_1$ is a caterpillar tree. Moreover, if we let $u'$ denote the $\preceq_N$-minimal vertex
     of $P^1$ that satisfies $\eta\prec_N u'$, then the leaf $v$ must, by construction and since $N$ is phylogenetic, be a child of $u'$ in $T_1$ and hence,
    $v$ is part of the cherry of $T_1$. By construction, the leaf set of $T_1$ is precisely $W_1$
    and one easily verifies that $\lca_{T_1}(x,x') = \lca_{N}(x,x')$
    for all $x,x'\in W_1$. Since $t_1$ is the restriction of $t$ to $T_1$,
    $t_1(\lca_{T_1}(x,x')) = t(\lca_{N}(x,x'))$ holds for all $x,x'\in W_1$. 
    Hence,  $(T_1, t_1)$ explains $\SM_1$. 
    Since $(N,t)$ is quasi-discriminating, it follows that
    the root $\rho_{T_1}$ of $T_1$ (which is the child $u$ of $\rho_N$ in $N$) satisfies
    $t_1(\rho_{T_1})=t(u)\neq t(\rho_N) =k$. Thus,  $\SM_1$ 
    satisfies  Definition~\ref{def:polarcat}(ii).  By analogous argumentation, 
    $\SM_2$ satisfies  Definition~\ref{def:polarcat}(ii).
  
    To summarize, $\SM$ is a polar-cat. To see that $\SM$ is handsome, observe that Lemma \ref{lem:morerainbowtriangle} implies
    $|X|\geq 3$. Assume that $|X|=3$. In this case, 
	\ref{def:strong2} implies that three colors must appear along the three pairs of vertices 
	in $X$ and $\SM$ is, therefore, a rainbow-triangle. In summary, $|X|\geq 3$ implies that 
	$\SM$ is a rainbow-triangle or that $|X|\geq 4$ and therefore, that $\SM$ is handsome.
\end{proof}

Given a galled-tree $(N,t)$, constructing the strudigram $\mathcal{S}(N,t)$ is straightforward. 
The proof of Lemma \ref{lem:seqdGaT=>pc} focuses on showing certain properties of $\mathcal{S}(N,t)$
for a specific type of galled-tree, rather than proving its existence. 
Proving the existence of a galled-tree that explains a given strudigram is
generally more challenging. However, the structured nature of polar-cats allows for the explicit
construction of a galled-tree that explains it. The main idea is to connect the caterpillar trees
that explain the the substrudigrams $\SM_1$ and $\SM_2$ of a
$(v,\SM_1,\SM_2)$-polar-cat with an additional root-vertex and by merging their
respective occurrences of the leaf $v$ under a common hybrid-vertex. This
procedure is formalized as follows (c.f. \cite[Def. 4.14]{HS:22}).

\begin{definition}\label{def:N(pc)}	
    Let $\SM$ be a $\join_k$-$(v,\SM_1,\SM_2)$-polar-cat. Let
    $(T_1,t_1)$, resp., $(T_2,t_2)$ be the discriminating caterpillar tree that
    explains $\SM_1$, resp., $\SM_2$ and where $v$ is a part of the cherry
    in both $T_1$ and $T_2$. We may, without loss of generality, assume that $T_1$ and $T_2$ are
    vertex-disjoint except for the leaf $v$. However, to distinguish the
             occurrences of $v$ in the two trees, we call $v$ in $T_i$ simply
             $v_i$ for $i\in\{1,2\}$.
    Then, the di-graph
    $N\coloneqq\N(v,\SM_1,\SM_2)$ is constructed as follows:\smallskip
    \begin{enumerate}
        \item Let $N'$ be the disjoint union of $T_1$ and $T_2$.
          \item
               Add a new root vertex $\rho_N$ to $N'$ along with edges 
              $(\rho_N,\rho_{T_1})$ and $(\rho_N,\rho_{T_2})$ to obtain $N''$. 
        \item Identify the leaves $v_1$ and $v_2$ of $N''$ into a new hybrid
              vertex denoted $\eta_N$, add a new occurrence of the leaf $v$ and
              the edge $(\eta_N, v)$ which results in the di-graph $N$.  \smallskip
    \end{enumerate}
    The labeling $t\coloneqq t(v,\SM_1,\SM_2)$ of $\N(v,\SM_1,\SM_2)$ is defined, for every inner vertex $u$ of $\N(v,\SM_1,\SM_2)$, by
    \[t(u)\coloneqq
        \begin{cases}
        t_1(u) &\text{if }u\in V^{0}(T_1) \\
        t_2(u) &\text{if }u\in V^{0}(T_2) \\
        k      &\text{if }u\in\{\rho_N,\eta_N\}.
    \end{cases}
    \]
\end{definition}

The label $t(\eta_N)$ may be arbitrarily defined: for definiteness we have chosen $t(\eta_N)=k$
here. We argue now that $(N,t)\coloneqq (\N(v,\SM_1,\SM_2), t(v,\SM_1,\SM_2))$ is an elementary,
quasi-discriminating, and leaf-separated network. One easily verifies that $N$ is a galled-tree consisting of a single
gall rooted at $\rho_N$ and terminating at $\eta_N$, since the trees $T_1$ and $T_2$ used in the
construction of $N\coloneqq \N(v,\SM_1,\SM_2)$ are caterpillars. Note that $\eta_N$ has a single
child, namely, the leaf $v$. Hence, $N$ is leaf-separated.
Furthermore, $\rho_N$ has no leaf-children and all other inner vertices have
precisely one leaf-child. In other words, $N$ elementary. One easily verifies that the labeling
$t=t(v,\SM_1,\SM_2)$ is well-defined. Moreover, since both
$(T_1,t_1)$ and $(T_2,t_2)$ are discriminating and since $t_i(\rho_{T_i})\neq k$, $1\leq i\leq 2$,
and since the only newly introduced inner vertex in $N$ that is distinct from $\rho_N$ is the hybrid
vertex $\eta_N$, it follows that $(N,t)$ is quasi-discriminating, see Fig.~\ref{fig:defNpc} for an
illustrative example. We summarize the latter findings in

\begin{observation}\label{obs:okDef1}
   If $\SM$ is a $(v,\SM_1,\SM_2)$-polar-cat, then the directed graph $\N(v,\SM_1,\SM_2)$ is a
   galled-tree, and $t(v,\SM_1,\SM_2)$ is a well-defined labeling of $\N(v,\SM_1,\SM_2)$. In
   particular, $(\N(v,\SM_1,\SM_2), t(v,\SM_1,\SM_2))$ is elementary, leaf-separated and quasi-discriminating. 
\end{observation}

\begin{figure}
	\centering
	\includegraphics[width=0.8\textwidth]{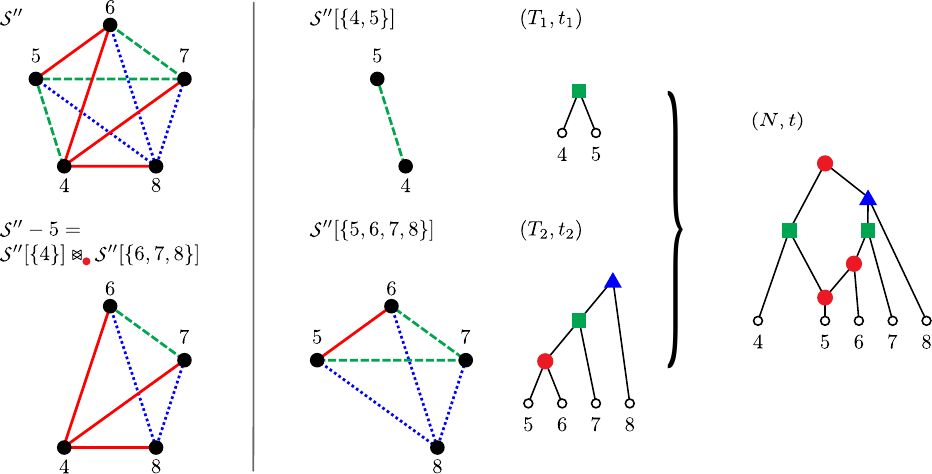}
	\caption{Top-left, a strudigram $\SM''$ is shown,  see Fig.~\ref{fig:exmpl-sketch1} for the color legend. 
					A network $(N'',t'')$ that explains $\SM''$ is provided already in Fig.~\ref{fig:exmapl1-pvr}. 
					Put $\SM''_1 \coloneqq \SM''[{4,5}]$, $\SM''_2 \coloneqq \SM''[{5,6,7,8}]$ and $v\coloneqq 5$.
					The strudigram  $\SM''-v$ is a $k$-series strudigram as $\SM''-v = (\SM''_1-v) \protect\join_k (\SM''_2-v)$  
					where  $k$ refers to the edge-color
					solid-red in the drawing. Hence, $\SM''$ satisfies Def.~\ref{def:polarcat}(i).
					The substrudigrams $\SM''_1$ and $\SM''_2$ can be explained by the
					discriminating caterpillar $(T_1,t_1)$ and $(T_2,t_2)$, respectively. In addition, 
					the vertex $v=5$ is a cherry in both $T_1$ and $T_2$ and,
					$t_1(\rho_{T_1})\neq k$ and $t_2(\rho_{T_2})\neq k$. Therefore, $\SM''$ satisfies Def.~\ref{def:polarcat}(ii). 
					In summary, $\SM''$ is a $(v,\SM''_1,\SM''_2)$-polar-cat. Even more, $\SM''$ is handsome. 
					The network $(N,t)$ is the network $(\N(v,\SM''_1,\SM''_2), t(v,\SM''_1,\SM''_2))$ as specified in 
					Def.\  \ref{def:N(pc)}. This network $(N,t)$ explains $\SM''$ and thus serves as a much simpler alternative
				   explanation of $\SM''$ compared to $(N'',t'')$ shown in Fig.~\ref{fig:exmapl1-pvr}.
	}			
	\label{fig:defNpc}
\end{figure}

Let us now investigate the structure of the unique gall $C$ in $(N,t)\coloneqq (\N(v,\SM_1,\SM_2),
t(v,\SM_1,\SM_2))$. Recall that both $\SM_1$ and $\SM_2$ have at least two vertices each, i.e., the
corresponding trees $T_1$ and $T_2$ have both at least two leaves. Hence, each side of the gall $C$
of must have at least one vertex that is distinct from $\rho_C$ and $\eta_C$ and, therefore, always
satisfies \ref{def:strong1}. In case the polar-cat has at least four vertices then we can conclude
that at least one of $\SM_1$ and $\SM_2$ has at least three vertices, implying that at least one of
$T_1$ and $T_2$ has at least two inner vertices. In this case, $C$ trivially satisfies
\ref{def:strong2} and is, therefore, strong. If a polar-cat $\SM$ is a rainbow-triangle, then
it easily seen that $(N,t)$ must contain at least three distinct labels, namely $|\{t(\rho_N),t(\rho_{T_1}),t(\rho_{T_2})\}|=3$, so that $C$ again satisfies \ref{def:strong2}. See also Figure~\ref{fig:rainbowT}.
The latter discussion together with Obs.\ \ref{obs:okDef1} implies
\begin{observation}\label{obs:okDef2}
   If $\SM$ is a handsome $(v,\SM_1,\SM_2)$-polar-cat, then $(\N(v,\SM_1,\SM_2), t(v,\SM_1,\SM_2))$ 
   is a strong, elementary and quasi-discriminating labeled galled-tree.
\end{observation}

The attentive reader may have noted that we could have simply identified the two occurrences of $v$
in Definition~\ref{def:N(pc)} without the additional leaf-child of $\eta_N$ and still obtained an elementary and
quasi-discriminating galled-tree. However, it would not be leaf-separated, a property that will
simplify some of the upcoming proofs.

\begin{lemma}\label{lem:pc=>gatex}
    Every $(v, \SM_1,\SM_2)$-polar-cat is explained by the labeled galled-tree
    $(\N(v,\SM_1,\SM_2), t(v,\SM_1,\SM_2))$.
\end{lemma}
\begin{proof}
  Let $\SM=(X,\Upsilon,\sigma)$ be a $(v,\SM_1,\SM_2)$-polar-cat and put
  $(N,t)\coloneqq(\N(v,\SM_1,\SM_2), t(v,\SM_1,\SM_2))$. Let $u_1$ and $u_2$ be the two children of
  $\rho_N$. Note that $N(u_1)$ and $N(u_2)$ are precisely the two caterpillar trees $(T_1,t_1)$
  and $(T_2,t_2)$ that explain $\SM_1$ and $\SM_2$, respectively. Let $W_1$ and $W_2$ be the vertex set
  of $\SM_1$ and $\SM_2$, respectively. 
    
   Let $x$ and $y$ be vertices of $\SM$ and, therefore, leaves in $N$. Suppose first that $x\in
  W_i\setminus \{v\}$ and $y\in W_j\setminus \{v\}$ with $\{i,j\}=\{1,2\}$. Since $\SM -v = (\SM_1
  -v) \join_k (\SM_2-v)$, we have $\sigma(xy)=k$. Moreover, $x$ and $y$ are leaf-children of vertices that are
  located on different sides of the unique gall in the elementary network $N$. Thus,
  $\lca_N(x,y)=\rho_N$. By construction, $t(\rho_N)=k$ and therefore, $t(\lca_N(x,y))=k$. 
  Suppose now that $x$ and $y$ are distinct elements of $W_i$ for some $i\in \{1,2\}$. In this case, $x,y$ are both leaves of $N(u_i)$. Note that
  $\parent_N(x)\neq \parent_N(y)$ and we may assume w.l.o.g.\ that
  $\parent_N(y)\prec_N\parent_N(x)$. It is easy to verify that $u\coloneqq \lca_N(x,y) = \lca_{T_{i}}(x,y)
  = \parent_N(x)$. Since the labeling $t$ restricts to $t_i$ on the inner vertices of $N(u_i)$ and
  since $(T_i,t_i)$ explains $S_i$, we thus have $t(u)=t_i(u)=\sigma(xy)$. In summary, $(N,t)$
  explains $\SM$.
\end{proof}

We now investigate the structures of modules in polar-cats.

\begin{lemma}\label{lem:pc=>primitive}
    If $\SM$ is a handsome polar-cat, then $\SM$ is truly-primitive.
\end{lemma}
\begin{proof} 
	Let $\SM=(X,\Upsilon,\sigma)$ be a handsome $(v, \SM_1,\SM_2)$-polar-cat. If $\SM$ is a 
	rainbow-triangle, then it is truly-primitive by Observation~\ref{obs:rainbowtriangle}.
	We may thus assume $|X|\geq4$. By
	Lemma~\ref{lem:pc=>gatex}, $\SM$ is explained by the labeled galled-tree $(N,t)\coloneqq
	(\N(v,\SM_1,\SM_2), t(v,\SM_1,\SM_2))$. By Observation~\ref{obs:okDef1} and \ref{obs:okDef2}, $(N,t)$ is both strong,
	elementary, leaf-separated and quasi-discriminating. Let $C$ denote the unique underlying gall of $N$ and let
	$\rho$ and $\eta$ denote its root and unique hybrid vertex, respectively. Let $h$ be the unique
	leaf in $N$ that is adjacent to $\eta$. In what follows, we let $\tilde w\coloneqq \parent_N(w)$
	denote the unique parent of a leaf $w$ in $N$. Since $N$ is elementary, $\tilde w$ is located on
	$C$ for every leaf $w$ in $N$.  
	
   Let $M\subseteq X$ be a set of vertices such that $2\leq|M|<|X|$, and put $\overline M\coloneqq X\setminus M$. Assume, for contradiction, that
   $M$ is a module of $\SM$. To recall, $X$ is the leaf set of $N$. Moreover, since $\preceq_N$ is a
   partial order and each leaf has a distinct parent, we always have that $\tilde v\prec_N \tilde
   w$, $\tilde w\prec_N \tilde v$ or $\tilde v$ and $\tilde w$ are  $\preceq_N$-incomparable for any two distinct $v,w\in X$.
   Moreover, no leaf has $\rho$ as a parent. We show first that the following case cannot occur.

\begin{owndesc}
	\item[\textnormal{\textit{Case A: There exist some $x,y\in M$ and some $u \in \overline M$ such
	that $\tilde x\succ_N \tilde u\succ_N \tilde y$.}}] We may, without loss of generality, assume
	that $x,y\in M$ are taken so that there is no $z\in M$ with $\tilde x\succ_N \tilde z\succ_N
	\tilde y$ and such that $\tilde u$ is the (unique) non-leaf child of $\tilde x$. Since $\tilde
	u\succ_N \tilde y$, we have $\tilde u\neq \eta$. This and the fact that $(N,t)$ is
	quasi-discriminating implies that $t(\tilde x)\neq t(\tilde u)$. Moreover, since $\tilde x\succ_N
	\tilde u\succ_N \tilde y$, we have $\lca_N(x,u)=\tilde x$ and $\lca_N(y,u)=\tilde u$. Therefore,
	$\sigma(xu)=t(\tilde x)\neq t(\tilde u) =\sigma(uy)$ which implies that $M$ is not a module; a
	contradiction. 
\end{owndesc}

\noindent
	Hence, Case A cannot occur under the assumption that $M$ is a module. We now distinguish between
	the cases that the unique child $h$ of $\eta$ satisfies (I) $h\in M$ or (II) $h\notin M$. 

\begin{owndesc}
	\item[\textnormal{\textit{Case (I): $h\in M$.}}] Since $|M|>1$, there must be a leaf $z\in
	M\setminus \{h\}$. In particular, $\tilde z\succ_N \eta$ holds. This and the fact that Case A
	cannot occur implies that, for the unique parent $\tilde x$ of
	$\eta$ in $N$ with $\tilde x \preceq_N \tilde z$, we have $x\in M$ (note, $x=z$ might be
	possible). Moreover, $|M|<|X|$ implies that there is some $u\in \overline M$.
	
	Suppose first that $\tilde u$ and $\tilde x$ are comparable for all $u\in \overline M$. Let
	$P^1$ be the side of $C$ that contains $\tilde x$ and $P^2$ be the other side of $C$. Note that
	$P^2$ contains at least one vertex distinct from $\rho$ and $\eta$, since $N$ is strong. Since
	all vertices $\tilde y \neq \rho, \eta$ along $P^2$ are incomparable with $\tilde x$, 
	since $\eta=\tilde h$ and $h\in M$ and since no leaf of $N$ is adjacent to $\rho$, we have $y \in M$ for all
	$\tilde y$ along $P^2$. Since Case A cannot occur, it follows that $v\in \overline M$ for the
	unique child $\tilde v$ of $\rho$ on $P^1$. Since $(N,t)$ is quasi-discriminating and since
	$\tilde v\neq \eta$, we have $t(\rho)\neq t(\tilde v)$. This together with $ \lca_N(x,v)=\tilde
	v$ and $\lca_N(y,v)=\rho$ implies $\sigma(xv)\neq \sigma(yv)$; a contradiction to $M$ being a
	module and $x,y\in M$ and $v\in \overline M$.
	Consequently, there is some $u\in \overline M$ such that $\tilde u$ and $\tilde x$ are
	incomparable. Since $u\succ_N \eta$, $h\in M$ and since Case A cannot occur, it follows
	that the unique child $\tilde v$ of $\rho$ with $\tilde v \succeq_N \tilde u$ satisfies $v \in
	\overline M$. Again we have $t(\rho)\neq t(\tilde v)$. By construction, $\lca_N(h,v) = \tilde v $
	and $\lca_N(x,v)=\rho$. Since $x,h\in M$ and $v\in\overline M$ we have $\sigma(hv)=\sigma(xv)$.
	However, $\sigma(hv) = t(\tilde v)$ and $\sigma(xv) = t(\rho)$ imply $\sigma(hv)\neq \sigma(xv)$;
	a contradiction. 
\end{owndesc}

\noindent
Thus, neither Case A nor (I) can occur if $M$ is a module. Hence, we consider the final
\begin{owndesc}
	\item[\textnormal{\textit{Case (II): $h\notin M$.}}]
	One easily verifies that $\tilde x = \lca_N(x,h) $ and $\tilde y = \lca_N(y,h)$ holds for all
	$x,y\in M\subseteq X\setminus \{h\}$. Hence, $t(\tilde x) = \sigma(xh)=\sigma(yh)=t(\tilde y)$
	must hold for all $x,y\in M$. Note that $\tilde x\neq \rho, \eta$ for all $x\in M$. The latter
	two arguments imply together with the fact that $(N,t)$ is quasi-discriminating that none of the
	elements $\tilde x$ and $\tilde y$ can be adjacent in $N$. Hence, if there are two $x,y\in M$
	such that $\tilde x\succ_N \tilde y$, then there is a vertex $\tilde u$ on $C$
	that satisfies $\tilde
	x\succ_N \tilde u \succ_N \tilde y$ and $t(\tilde u)\neq t(\tilde x)$. Consequently, $u\in
	\overline M$ and we have Case A; a contradiction. Thus, any two $x,y\in M$ must be incomparable
	in $N$. One easily verifies that, in case $|M|\geq 3$, there are two $x,y\in M$ such $\tilde
	x\succ_N \tilde y$ since $N$ is elementary. This and $|M|>1$ implies that $|M|=2$. Let 
	$M=\{x,y\}$. Since $N$ is strong and $|X|\geq 4$ and $\tilde x$ and $\tilde y$
	are incomparable, there is some $\tilde u$ on $C$ that is comparable to either $\tilde x$ or $\tilde y$. 
	
	Suppose first that $\tilde u\succ_N \tilde x$ and thus, $\lca_N(x,u)=\tilde u$. As Case A cannot occur, we
	may w.l.o.g.\ assume that $\tilde u$ is a child of $\rho$. Since $(N,t)$ is quasi-discriminating we
	have $t(\tilde u)\neq (\rho)$. Since $\tilde u$ and $\tilde y$ are incomparable, we have
	$\lca_N(u,y)=\rho$. Hence, $\sigma(xu)=t(\tilde u)\neq (\rho)=\sigma(uy)$ together with
	$u\notin M$ yields a contradiction to $M=\{x,y\}$ being a module. Therefore, $\tilde u\succ_N
	\tilde x$ and, by similar arguments, $\tilde u\succ_N \tilde y$ is not possible for any $u\in
	\overline M$. This and $M=\{x,y\}$ implies that $\tilde x$ and $\tilde y$ must be the two
	children of $\rho$. In this case, there is some $u\in\overline M$ such that $\tilde x\succ_N \tilde u$ and thus, $\lca_N(x,u)=\tilde x$.
	Since $\tilde x\neq \eta$ is a child of $\rho$, we have $t(\tilde x)\neq t(\rho)$. Moreover,
	$\tilde y$ and $\tilde u$ are incomparable and, thus, $\lca_N(u,y)=\rho$. Hence,
	$\sigma(uy)=t(\rho)\neq t(\tilde x)=\sigma(xu)$ yields again a contradiction. 
\end{owndesc}

\noindent
   In summary, there is no $M\subseteq X$ with $2\leq|M|<|X|$  and such that $M$ is a module. 
   Hence $\SM$ consists of trivial modules only, i.e., $\SM$ is primitive. Since $|X|\geq 4$, 
   $\SM$ is truly-primitive. 
\end{proof}

\begin{lemma}\label{lem:module-at-gall}
	Let $(N,t)$ be a galled-tree that contains a gall $C$. If $u$ and $v$ are the two 
	children of $\rho_C$ that are located on $C$, then $L(N(u))\cup L(N(v))$ is a module of 
	$\mathcal{S}(N,t)$ containing at least two vertices.
\end{lemma}
\begin{proof}
	Let $u$ and $v$ be two 	children of $\rho_C$ that are located on $C$ and 
	put $M\coloneqq L(N(u))\cup L(N(v))$. Let $\SM =(X,\Upsilon, \sigma)$ be the strudigram 
	$\mathcal{S}(N,t)$. If
	$M=X$, then $M$ is trivially a module of $\SM$. Otherwise, there is a vertex $y\in X\setminus
	M$. In particular, $y$ is a leaf in $N$ and thus, neither a descendant of $u$ nor of
	$v$. Hence, $\lca_{N}(x,y)\succeq_N\rho_C$ for every $x\in M$. Consequently,
	$\sigma(xy)=t(\lca_N(x,y))=t(\lca_N(\rho_C,y))=t(\lca_N(x',y))=\sigma(x'y)$ for all $x,x'\in
	M$  and any $y\notin M$. Thus, $M$ is a module of $\SM$.

	We consider now the size of $M$. Let $\eta_C$ be the unique hybrid  of $C$. 
	Since $u$ and $v$ are contained in the gall $C$ and $\eta_C$ is the
	unique $\preceq_C$-minimal vertex in $C$ it follows that $\eta_C\preceq_N u,v$.
	By \cite[L.\ 17]{Hellmuth2023}, $L(N(\eta_C))\subseteq	L(N(v))$ and 
	$L(N(\eta_C))\subseteq	L(N(u))$ must hold.
	We now show that
	$M\setminus L(N(\eta_C))\neq \emptyset$. This statement is clearly satisfied if $L(N(\eta_C))$
	is a proper subset of both $L(N(u))$ and of $L(N(v))$. Hence, suppose that
	$L(N(u))=L(N(\eta_C))$. It is an easy task to verify that $u=\eta_C$, and $v\neq\eta_C$. The latter
	implies that $v$ has at least one child $v'$ not on $C$, so that $\emptyset\neq
	L(N(v'))\subset L(N(v))$. In particular, $L(N(v))\setminus L(N(\eta_C))\neq\emptyset$.
	Analogously, $L(N(v))=L(N(\eta_C))$ implies $L(N(u))\setminus L(N(\eta_C))\neq\emptyset$. In
	both cases, we thus have that \[M\setminus L(N(\eta_C)) = (L(N(u))\setminus
	L(N(\eta_C)))\cup(L(N(v))\setminus L(N(\eta_C)))\neq \emptyset.\] To summarize, we have
	$|M|=|M\setminus L(N(\eta_C))|+|L(N(\eta_C))|\geq 1 + 1=2$.
\end{proof}

\begin{lemma}\label{lem:primGatex=>eqdGaT}
	If $\SM$ is a truly-primitive and \textup{\gatex} strudigram,
    then every labeled galled-tree $(N,t)$ that explains $\SM$ is strong, elementary and quasi-discriminating. 
\end{lemma}
\begin{proof}
	Let $\SM=(X,\Upsilon,\sigma)$ be a truly-primitive \gatex strudigram. Hence, $\SM$ is
	explained by some labeled galled-tree $(N,t)$. Since $\SM$ is
	truly-primitive, we have $|X|\geq 3$. By Theorem~\ref{thm:unp<=>treeExpl} and since $\SM =
	\SM[X]$ is primitive, $\SM$ cannot be explained by a labeled tree and, therefore, $N$ contains at
	least one gall $C$. Let $u$ and $v$ denote the two children of $\rho_C$ that are located on $C$.
	By Lemma~\ref{lem:module-at-gall} the set $M = L(N(u))\cup L(N(v))$ is a module of $\SM$ such
	that $|M|>1$. Since $\SM$ is primitive it has no nontrivial modules, which enforces $M=X$. This
	means that $\rho_C=\rho_N$, and $\rho_N$ does not have any other children than $u$ and $v$. In
	particular, $C$ is the only gall of $N$.

   We continue with showing that $N$ is elementary. Let $W = V(C)\setminus \{\rho_C\}$. For each
   $w\in W$ we introduce the set \[\mathcal{L}(w)\coloneqq\left\{x\in L(N(w)) \mid \forall u\in W
   \text{ with } u\prec_N w \text{ we have } x\npreceq_N u\right\}.\] In simple words,
   $\mathcal{L}(w)$ consists of the leaves in $N(w)$ which are not descendant of any $u\in W$ with
   $u\prec_N w$. Note that each vertex $w\in W\setminus\{\eta_C\}$ has a child $v_w$ that is not located on $C$. In particular, $v_w$
   is incomparable to each $u\in W$ with $u\prec_N w$. Hence, $L(N(v_w))\subseteq \mathcal{L}(w)$
   and, therefore, $|\mathcal{L}(w)|\geq 1$ for all $w\in W\setminus\{\eta_C\}$. 
	Moreover, it is clear that $|\mathcal{L}(\eta_C)|=|L(N(\eta_C))|\geq1$.
   By construction,
   $\mathcal{L}(w)\cap\mathcal{L}(w')=\emptyset$ for all distinct $w,w'\in W$. This and
   $|\mathcal{L}(w)|\geq 1$ for all $w\in W$ implies that $|L(N)|=|X|>|\mathcal{L}(w)|\geq 1$ for
   all $w\in W$. Moreover, since $\rho_N = \rho_C$ has out-degree two, $\{ \mathcal{L}(w) \mid w\in
   W\}$ forms a partition of $L(N)$. Assume, for contradiction, that $|\mathcal{L}(w)|>1$ for some
   $w\in W$. Put $M=\mathcal{L}(w)$. Clearly, $\lca_N(M)\preceq_N w$ must hold. It is an easy task
   to verify that $\lca_N(x,u) = \lca_N(x',u) = \lca_N(w,u)\in V(C)$ and thus, $t(\lca_N(x,u)) =
   t(\lca_N(x',u))$ for every $u\in L(N)\setminus M$ and $x,x'\in M$. Hence, $M$ is a module of
   $\SM$ satisfying $1<|M|<|X|$; contradicting primitivity of $\SM$. Hence, $|\mathcal{L}(w)|=1$
   must hold for all $w\in W$ and since $N$ does not contain vertices with in- and out-degree 1,
   each $w\in W\setminus\{\eta_C\}$ has precisely one leaf-child with in-degree one. 
   Moreover, $|\mathcal{L}(\eta_C)|=1$ implies that $\eta_C$ itself is a leaf, or $\eta_C$ is an inner-vertex with precisely one child, and that child is a leaf.
   This together with the latter arguments implies that
   $N$ is elementary.
    
	We continue with showing that $(N,t)$ is a strong galled-tree, i.e. that $C$ is a strong gall
	and thus satisfies \ref{def:strong1} and \ref{def:strong2}. Let $P^1$ and $P^2$ be the two sides of $C$.
	In what follows, we let $\widetilde V(P^i)$ denote the set of vertices of $P^i$ 
	that are distinct from its end vertices, $i\in\{1,2\}$. 

	Assume, for contradiction, that $C$ does not satisfy \ref{def:strong1}. W.l.o.g., assume that
	$\widetilde  V(P^1)=\emptyset$, i.e., $P^1$ consists of $\rho_C$ and $\eta_C$ only. Consider the unique
	leaf $x$ of $N$ that satisfy $x\preceq\eta_C$ and the leaf-child $y\neq\eta_C$ of the $\preceq$-minimal vertex $v$ in $\widetilde V(P^2)$.
	Since $N$ is elementary and has $|X|\geq 3$ leaves we have $|M|<|X|$ for $M=\{x,y\}$. Moreover,
	for any leaf $u\neq x,y$ in $N$ we have $\parent_N(u)\succ v,\eta_C$ and, in particular,
	$\lca_N(x,u)=\parent_N(u)= \lca_N(y,u)$ and therefore, $t(\lca_N(x,u))= t(\lca_N(y,u))$ for all
	$u\in X\setminus M$. Thus, $M$ is a non-trivial module of $\SM$; a contradiction. Hence $C$ must
	satisfy \ref{def:strong1}.
	
	Note that since $N$ is elementary and $C$ satisfies \ref{def:strong1}, Condition \ref{def:strong2} is trivially satisfied whenever $|X|>3$.
	We thus consider the case $|X|=3$, which together with \ref{def:strong1} and $N$ being elementary
	implies that $\widetilde V(P^i)=\{v_i\}$ i.e. $|\widetilde V(P^i)|=1$ for $i\in \{1,2\}$ must hold.
	Moreover, for $i \in \{1,2\}$, the vertex $v_i$ has two children; namely a leaf $x_i$ and the hybrid $\eta$
	of $C$. Finally, $N$ must contain a unique leaf $x_3$ such that $x_3\preceq \eta$, so that
	$X=\{x_1,x_2,x_3\}$. As $\SM$ is a truly-primitive strudigram with three vertices it is, by
	Lemma~\ref{lem:morerainbowtriangle} a rainbow triangle. One now easily verifies that
	$|\{t(\rho_C),t(v_1),t(v_2)\}|=|\{\sigma(x_1x_2),\sigma(x_1x_3),\sigma(x_2x_3)\}|=3$, i.e. $C$
	satisfies \ref{def:strong2}. In summary so far, $(N,t)$ is a strong, elementary network.

   Finally, assume for contradiction that $(N,t)$ is not quasi-discriminating and let $v,v'\in
   V(C)\setminus \{\eta_C\}$ be adjacent tree vertices of $C$ such that $t(v)=t(v')$. Without loss
   of generality, we may suppose $\rho_N\succeq_N v\succ v'\succ\eta_C$. 
    Let $x'\neq \eta_C$ denote the leaf-child of $v'$ with in-degree one.
   First assume $\rho_N=v$ and thus that $v$ has no leaf-children. It
   is an easy task to verify that, in this case, $\lca(x',y) \in \{v',v\}$ for all $y\in M \coloneqq
   X\setminus \{x'\}$. Since $t(v) = t(v')$ and $2\leq |X|-1 = |M|<|X|$, $M$ is a non-trivial module
   of $\SM$; a contradiction. Hence, $v\neq \rho_N$ must hold. 
   Since $\rho_N\succ v\succ v'\succ\eta$ and $N$ is elementary, $v$ has a unique leaf-child $x$. Moreover, 
   since $v$ and $v'$ are adjacent, we have for $y\in X\setminus\{x,x'\}$ either (i)
   $\parent_N(y)\succ_N v,v'$ or (ii) $v,v'\succ_N \parent_N(y)$ or (iii) $\parent_N(y)$ is
   incomparable to both $v$ and $v'$. One easily verifies that, in Case (i) we have $\lca(x,y)
   =\lca(x',y) = \parent(y)$, in Case (ii) we have $\lca(x,y) = v$ and $\lca(x',y) = v'$ and in Case
   (iii) we have $\lca(x,y) = \lca(x',y) = \rho_N$. Hence, in all cases it holds that $t(\lca(x,y))
   =t(\lca(x',y))$ which implies that $M'=\{x,x'\}$ is a non-trivial module of $\SM$; a
   contradiction. Thus, $(N,t)$ must be quasi-discriminating. 
\end{proof}

We now state one of the main result of this section.

\begin{theorem}\label{thm:primGatex}
    For every strudigram $\SM$, the following statements are equivalent.\smallskip
    \begin{enumerate}
        \item[(a)] $\SM$ is explained by a strong, elementary and quasi-discriminating galled-tree.
        \item[(b)] $\SM$ is a handsome polar-cat.
		\item[(c)] $\SM$ is a truly-primitive polar-cat.
        \item[(d)] $\SM$ is truly-primitive and \textup{\gatex}.
    \end{enumerate}
\end{theorem}
\begin{proof}
	If $\SM=(X,\Upsilon,\sigma)$ can be explained by a  strong, elementary and quasi-discriminating labeled-galled-tree, 
	then we can apply Lemma~\ref{lem:seqdGaT=>pc} to conclude that $\SM$ is a handsome polar-cat, so (a) implies (b).
	Assume that  $\SM$ is handsome polar-cat. By Lemma \ref{lem:pc=>primitive}, $\SM$ is truly-primitive and, thus, (b) implies (c).
	By Lemma~\ref{lem:pc=>gatex}, $\SM$ is \gatex. Thus, (c) implies (d).
	Finally, by Lemma \ref{lem:primGatex=>eqdGaT}, (d) implies (a), which completes this proof.
\end{proof}

By Theorem \ref{thm:primGatex}, handsome polar-cats form a subclass of truly-primitive strudigrams. 
This  together with Theorem \ref{thm:unp<=>treeExpl} implies 
\begin{corollary}\label{cor:TRS-HPC=>P4T}
	Every truly-primitive strudigram and thus, every handsome polar-cat contains 
	an induced $P_4$ or a rainbow-triangle.
\end{corollary}

By Theorem \ref{thm:primGatex}, all truly-primitive \gatex strudigrams can be explained by a
(strong, elementary, and quasi-discriminating) galled-trees and are exactly the handsome polar-cats.
Moreover, for each prime module $M$ of some strudigram $\SM$ for which $|M|>1$, Thm.\
\ref{thm:Mprime=>QuotientMtruly-primitive} implies that $\SM' = \SM[M]/\Mmax(\SM[M])$ is
truly-primitive. Hence, we may apply the machinery of pvr-networks in the context of galled-trees.

\begin{lemma}\label{lem:galled-tree-pvr}
	If $\pfam(\SM)$ is a prime-explaining family of a strudigram $\SM$ which contains only galled-trees,
	then the pvr-network $\pvr(\SM,\pfam(\SM))$ is a galled-tree.
\end{lemma}
\begin{proof}
	Let $\pfam(\SM)=\{(N_M,t_M)\mid M\in\P\}$ be a prime-explaining family of a strudigram
	$\SM=(X,\Upsilon,\sigma)$, where $\P$ is the set of prime-labeled vertices in the MDT $(\MDT,\tau)$ of $\SM$.
	Furthermore, put $(N,t)\coloneqq\pvr(\SM,\pfam)$ and assume that $N_M$ is a galled-tree for each
	$M\in\P$. By Lemma~\ref{lem:pvr-well-defined}, $N$ is a network on $X$. 
	
	To verify that $N$ is a galled-tree, we must verify that $N$ is phylogenetic and that every non-trivial biconnected component of $N$ is a gall. 
	Since $\MDT$ is phylogenetic and the in- and out-degrees of non-prime vertices in $\MDT$ remain unaffected when
	constructing $N$, it follows that all  vertices	$v$ in $\MDT$ that do not correspond to a prime module of $\SM$
	do not satisfy $\outdeg_N(v)=\indeg_N(v)=1$. Let us now consider a vertex $v$
	in $\MDT$ that corresponds to a prime module $M'$ of $\SM$. 
	In this case, $v$ is now the root of $N_{M'}$ and thus, $\outdeg_N(v)=\outdeg_{N_{M'}}(\rho_{N_{M'}})>1$ must hold. 	
	Moreover, since $N_{M'}$ is phylogenetic, none of its inner vertices $v$ satisfy in $N$: $\outdeg_N(v)=\indeg_N(v)=1$.
	If, on the other hand, $v$ is a leaf in $N_{M'}$, then it corresponds to some vertex in $\MDT$ which, by the latter arguments, 
	does not satisfy $\outdeg_N(v)=\indeg_N(v)=1$. In summary, $N$ is phylogenetic.
	
	Moreover, for any two $M',M\in\P$, the networks $N_M$ and $N_{M'}$ can, by construction, share at most one vertex $v$, 
	namely the root of one of these networks and a leaf of the other network. It follows that, in the latter case, 
	$v$ is a cut vertex. This immediately implies that $C$ is a non-trivial biconnected component in $N$ if and only if $C$ is a 
	non-trivial biconnected component in $N_M$ for precisely one $M\in \P$. This and the fact that all
	non-trivial biconnected components in each of the $N_M$ are galls, implies that all
	non-trivial biconnected component in $N$ are galls. Taking the latter arguments together, 
	$N$ is a galled-tree.
\end{proof}

\begin{figure}
	\centering
	\includegraphics[width=0.8\textwidth]{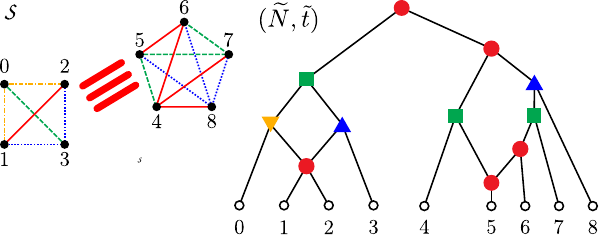}
	\caption{Shown is a labeled network $(\tilde N, \tilde t)$ that explains the strudigram $\SM$ as shown in Fig.~\ref{fig:exmapl1-pvr}. 
					 $(\tilde N, \tilde t)$ is the pvr-network obtained from the MDT of $\SM$ 
					 by replacing the prime vertices by the network $(N',t')$ as shown in Fig.~\ref{fig:exmapl1-pvr}
					 and the network $(N,t)$ as shown in Fig.~\ref{fig:defNpc}. Since $(\tilde N, \tilde t)$  is a galled-tree, 
					 $\SM$ is \gatex. 
				}
	\label{fig:gatex_pvr_8}
\end{figure}

As it turns out, \gatex strudigrams are precisely those strudigrams for which particular quotients are handsome polar-cats. 

\begin{definition}\label{def:primecat}
\PrimeCat denotes the set of all strudigrams $\SM$ for which $\SM[M]/\Mmax(\SM[M])$ is a handsome polar-cat
for every prime module $M$ of $\SM$ with $|M|>1$. 
\end{definition}

\begin{theorem}\label{thm:CharprimeCat}
The following statements are equivalent for every strudigram $\SM$
\begin{enumerate}[label=(\alph*)]
	\item $\SM$ is \textup{\gatex}
	\item $\SM[W]$ is \textup{\gatex} for all non-empty subsets $W\subseteq V(\SM)$.
	\item $\SM\in\PrimeCat$
\end{enumerate}
\end{theorem}
\begin{proof}
	Let $\SM = (X,\Upsilon,\sigma)$ be a strudigram. To prove that (a) implies (b), assume that $\SM$ is explained
	by a labeled galled-tree $(N,t)$ on $X$, and let $W \subseteq X$ be a non-empty subset. 
	Put $\mathfrak{C}'\coloneqq \{C \mid C\in \mathfrak{C}_N \text{ and } C \subseteq  W\}$.
	Since 	$N$ is a galled-tree, Theorem \ref{thm:galled-no3overlap} implies that 
	$\mathfrak{C}$ must be closed, and satisfies (L) and (N3O). Repeated application of Lemma \ref{lem:hereditary-cluster}
	to all $x\in X\setminus W$ shows that 
	$\mathfrak{C}'$ must be closed and satisfies (L) and (N3O).
	Moreover,  Theorem \ref{thm:galled-no3overlap} implies
	that $G\doteq \Hasse(\mathfrak{C}')$ is a galled-tree on $W$.
	By Lemma \ref{lem:lev1-klca}, $N$ is an lca-network. This together with Prop.~\ref{prop:lcaN-explain}
	implies that $G$ can be equipped with a labeling $t'$ such that  $(G,t')$ explains $\SM[W]$. 
	In summary, $S[W]$ is \gatex. Hence, (a) implies (b).

	To show that (b) implies (c), assume that $\SM[W]$ is \gatex for every non-empty subset $W\subseteq X$.
	If $\SM$  does not contain any prime module $M$ with $|M|>1$, then $\SM\in\PrimeCat$ is vacuously true.
	Assume that $\SM$ contains a prime module $M$ with $|M|>1$. Consider the quotient $\SM'\coloneqq
	\SM[M]/\Mmax(\SM[M])$ with $\Mmax(\SM[M]) = \{M_1 , \dots , M_k\}$. By Obs.\ \ref{obs:quotient},
	$\SM'\simeq \SM[W]$ with $W\subseteq M$ such that for all $ i \in \{1, \dots,k\}$ we have $|M_i
	\cap W | = 1$. By assumption, $\SM[W]$ is \gatex and, therefore,  $\SM'$ is \gatex.
	Since $M$ is a prime module, Thm.\ \ref{thm:Mprime=>QuotientMtruly-primitive} implies that
	$\SM'$ is truly-primitive. The latter two arguments together with Theorem~\ref{thm:primGatex}
	imply that $\SM'$ is a handsome polar-cat. Since the latter arguments hold for every prime
	module $M$ of $\SM$ with $|M|>1$, we have $\SM\in\PrimeCat$ i.e. (b) implies (c).

	Finally, assume that $\SM\in\PrimeCat$. Applying Theorem~\ref{thm:primGatex} ensures the
	existence of a (strong, elementary and quasi-discriminating) galled-tree $(N_M,t_M)$ that
	explains $\SM[M]/\Mmax(\SM[M]$) for every prime module $M$ of $\SM$ with $|M|>1$. Hence there
	exist a prime-explaining family $\pfam$ of $\SM$ containing only galled-trees. The pvr-network
	$(N,t)\coloneqq\pvr(\SM,\pfam)$ is a galled-tree by Lemma~\ref{lem:galled-tree-pvr} and, by
	Proposition~\ref{prop:pvr-explains-S}, $(N,t)$ explains $\SM$. Hence (c) implies (a), which
	completes the proof.
\end{proof}

In contrast to Prop \ref{prop:SM-hasExpl}, we obtain
\begin{corollary}
Every \textup{\gatex} strudigram  $\SM = (X,\Upsilon,\sigma)$ can be explained by a  network $(N,t)$ on $X$ with
$O(|X|)$ vertices.
\end{corollary}
\begin{proof}
Since  $\SM = (X,\Upsilon,\sigma)$ is \gatex, it can be explained by a galled-tree $(N,t)$ on $X$. 
By \cite[Prop.\ 1]{CRV:09}, $|V(N)| \leq 4|X | - 3\in O(|X|)$. 
\end{proof}


\section{Algorithmic Aspects}
\label{sec:algo}

In this section, we show that it can be verified whether $\SM$ is \gatex and, in the affirmative, 
that a labeled galled-tree that explains $\SM$ can be constructed in polynomial time.
Aided by Theorem~\ref{thm:CharprimeCat}, recognition of \gatex strudigrams is, in essence, solved by
recognition of polar-cats. For the latter, we first provide two further structural results.

\begin{lemma}\label{lem:handsome=>P4/RT}
	Let $\SM$ be a handsome polar-cat. Then, for every vertex $v$ for which 
	$\SM$ is a $(v,\SM_1,\SM_2)$-polar-cat it holds that
	$v$ is located on every induced $P_4$
	and every rainbow-triangle that exists in $\SM$.
\end{lemma}
\begin{proof}
Let  $\SM=(X,\Upsilon,\sigma)$ be a handsome polar-cat. 
By Cor.\ \ref{cor:TRS-HPC=>P4T}, $\SM$ contains an induced  $P_4$ or rainbow-triangle. 
Since $\SM$ is a polar-cat it is, by definition, in particular a $(v,\SM_1,\SM_2)$-polar-cat for some 
$v\in V(\SM)$. We show that $v$ is located on any $P_4$ and any rainbow-triangle that may exists in $\SM$.
By Lemma \ref{lem:pc=>gatex} and \ref{lem:primGatex=>eqdGaT}, 
$\SM$ is explained by the strong, elementary, quasi-discriminating and leaf-separated network 
$(N,t)\coloneqq (\N(v,\SM_1,\SM_2), t(v,\SM_1,\SM_2))$. 

By construction of $N=\N(v,\SM_1,\SM_2)$, $v$ is the single
child of the unique hybrid vertex $\eta$ in $N$.  
In particular $v$ is a leaf in $N$. 
Note that, by  \cite[L.\ 48]{Hellmuth2023}, if $L(N(u))$ and $L(N(w))$ overlap for some $u,w\in V(N)$, then $L(N(u)) \cap L(N(w)) = L(N(\eta))=\{v\}$. 
Thus, $\mathfrak{C}'\coloneqq \mathfrak{C}_N-v$ does not contain any overlapping clusters and is, therefore, 
a hierarchy. By \cite[Cor.\ 9]{Hellmuth2023}, $G\doteq \hasse(\mathfrak{C}')$ is a phylogenetic tree on $W \coloneqq X\setminus \{v\}$. 
By Proposition~\ref{prop:lcaN-explain}, 
$G$ can be equipped with a labeling $t'$ such that $(G,t')$ explains $\SM-v = S[W]$.
Hence, $\SM-v$ is explained by a labeled tree.
This together with Theorem \ref{thm:unp<=>treeExpl} implies that $\SM-v$ is $P_4$-free and rainbow-triangle-free. 
Hence, $v$ must be located on every $P_4$ and every rainbow-triangle that may exists in $\SM$.
\end{proof}

\begin{lemma}\label{lem:sets12same}
	Let $\SM$ be a $(v,\SM_1,\SM_2)$-polar-cat. Then, $\Mmax(\SM-v)=\{V(\SM_1-v),V(\SM_2-v)\}$. If $\SM$, in addition, is a
	$(v,\TM_1,\TM_2)$-polar-cat, then $\{\SM_1,\SM_2\}  = \{\TM_1,\TM_2\}$.
\end{lemma}
\begin{proof}
Assume that $\SM$ is a $(v,\SM_1,\SM_2)$-polar-cat. Note that, by definition of polar-cats, we have
$\SM-v=(\SM_1-v)\join_k(\SM_2-v)$ for some $k$. This, in particular, implies that 
$\SM'\coloneqq\SM-v$ is not a prime strudigram. By Theorem~\ref{thm:kseries=k'series}, the set $\Mmax(\SM')$ is a unique partition
of $V(\SM')$ such that $\SM'[M]$ is $k$-prime for each $M\in\Mmax(\SM')$. It is clear that
$\{V(\SM_1-v),V(\SM_2-v)\}$ is a partition of $V(\SM')=V(\SM)\setminus\{v\}$, and we proceed by
showing that $\SM_i-v$ is $k$-prime for $i=1,2$.

By definition, there is some discriminating caterpillar tree $(T,t)$ that explains $\SM_1$, such
that $t(\rho_{T})\neq k$. Since $T$ is a caterpillar tree, it has at least two leafs and $\rho_{T}$
has at least two children. If $\SM_1$ has only two vertices, then $\SM_1-v$ has precisely one vertex
and is vacuously $k$-prime. If, instead, $|V(\SM_1)|\geq3$, then $v$ is not the leaf-child of
$\rho_T$, since $v$ is part of the cherry of $T$. Thus, the leaf-child $x\neq v$ of $\rho_T$
satisfies $\sigma|_{V(\SM_1)}(xy)=t(\rho_T)\neq k$ for all $y\in V(\SM_1)\setminus\{x\}$.
Hence $\SM_1-v$ is not the $k$-join
of any two substrudigrams, and $\SM_1-v$ is $k$-prime. Similarly, $\SM_2-v$ is $k$-prime. In
conclusion, $\Mmax(\SM-v)=\{V(\SM_1-v),V(\SM_2-v)\}$.

Suppose now that there are strudigrams $\TM_1$ and $\TM_2$ such that $\SM$ is also 
$(v,\TM_1,\TM_2)$-polar-cat. By analogous reasoning, we
have $\Mmax(\SM-v)=\{V(\TM_1-v),V(\TM_2-v)\}$. Since $\Mmax(\SM-v)$ is uniquely determined, 
it holds that $\{\SM_1-v,\SM_2-v\}=\{\TM_1-v,\TM_2-v\}$. Since, in addition, 
$\SM_1, \SM_2, \TM_1$ and $\TM_2$ contain $v$ and are induced substrudigrams of $\SM$, it follows that 
$\{\SM_1,\SM_2\}=\{\TM_1,\TM_2\}$.
\end{proof}

To recognize polar-cats, we provide \texttt{Check\_polar-cat} in Algorithm~\ref{alg:checkPC}, and show its correctness.

\begin{algorithm}[tb] 
	\small 
	\caption{\texttt{Check\_polar-cat}}
	\label{alg:checkPC}
	\begin{algorithmic}[1]
	\Require  A truly-primitive strudigram $\SM$
	\Ensure  A labeled galled-tree $(N,t)$ that explains $\SM$ if $\SM$ is a (handsome) polar-cat, otherwise \texttt{false}
	\State Let $H$ be an induced $P_4$ or a rainbow-triangle in $\SM$ \label{alg:line:H}
	\ForAll{$v\in V(H)$} \label{alg:line:v}
		\State Compute $\SM'=\SM-v$ and the MDT $(\MDT,\tau)$ of $\SM'$.\label{alg:line: comp MDT}
		\If {$\MDT$ does not contain prime-labeled vertices} \label{alg:line: check MDT}
			\If {the root of $\MDT$ has precisely 2 children}	\label{alg:line: check child}
				\State Find $\Mmax(\SM')=\{M_1,M_2\}$ 
				\State Let $\SM_i = \SM[M_i\cup\{v\}]$ for $i=1,2$ \label{alg:line: def substrud}
				\If {$\SM_i$ satisfy Definition~\ref{def:polarcat}(ii) for $i=1,2$} \label{alg:line: check pc}
					\State Construct the galled-tree $\N(v,\SM_1,\SM_2)$ according to Definition~\ref{def:N(pc)}
					\State \Return $\N(v,\SM_1,\SM_2)$ \label{alg:line:ret galledtree}
				\EndIf
			\EndIf	
		\EndIf
	\EndFor 
	\State \Return \texttt{false} \label{alg:line:return-false}
	\end{algorithmic}
	\end{algorithm}

\begin{lemma}\label{lem: pol-check-pc}
	Let $\SM$ be a truly-primitive strudigram. In polynomial time, \textnormal{\texttt{Check\_polar-cat}} correctly verifies if $\SM$
	is a (handsome) polar-cat and, in the affirmative, constructs a labeled galled-tree that explains $\SM$.
\end{lemma}
\begin{proof}
	Let $\SM=(X,\Upsilon,\sigma)$ be a truly-primitive strudigram. Since $\SM$ is truly-primitive, we
	know that $n\coloneqq|X|\geq3$. Since Theorem~\ref{thm:CharprimeCat} ensures $\SM$ is a
	$(v,\SM_1,\SM_2)$-polar-cat if and only if it is a handsome $(v,\SM_1,\SM_2)$-polar-cat, it suffices
	to show that \texttt{Check\_polar-cat} returns \texttt{False} if and only if $\SM$ is not a
	polar-cat.
	By Lemma \ref{lem:handsome=>P4/RT}, every vertex $v$ for which $\SM$ is a
	$(v,\SM_1,\SM_2)$-polar-cat must be located on all induced $P_4$ and all rainbow-triangles that
	exist in $\SM$. By Cor.\ \ref{cor:TRS-HPC=>P4T}, at least one induced $P_4$ or rainbow-triangle
	exists in $\SM$ and is detected and denoted by $H$ in Line \ref{alg:line:H}. By the latter
	arguments, if $\SM$ is a polar-cat, then it is a $(v,\SM_1,\SM_2)$-polar-cat for at least one
	vertex $v$ in $H$. All vertices $v$ of $H$ are now examined in the for-loop starting in Line
	\ref{alg:line:v}.

	Observe that $\SM'\coloneqq\SM-v$ can be computed in polynomial time, and that the
	modular decomposition tree $(\MDT,\tau)$ of $\SM'$ can be constructed in $O(|V(\SM')|^2)=O(n^2)$ time
	\cite{EHREN1994}, as done in Line~\ref{alg:line: comp MDT}. For Line~\ref{alg:line: check MDT},
	the set $\P$ of prime vertices of $(\MDT,\tau)$ can be found by a polynomial-time traversal of
	$\MDT$. If $\P\neq\emptyset$, then $\SM'$ has at least one induced $P_4$ or rainbow triangle, by
	Theorem~\ref{thm:unp<=>treeExpl}. Thus $v$ does not lie on all $P_4$s and rainbow triangles of
	$\SM$, so that Lemma~\ref{lem:handsome=>P4/RT} implies that $\SM$ is not a
	$(v,\SM_1,\SM_2)$-polar-cat. The algorithm thus correctly continues to the next vertex of $H$.

	Otherwise, namely if $\P=\emptyset$, we next consider Line~\ref{alg:line: check child}. Since
	$n\geq3$, we have $|V(\SM')|\geq 2$, i.e., the tree $\MDT$ has at least two leafs.  Hence, 
	the root $\rho_\MDT$ has at least two children. By definitions of MDTs, we have
	$\child_\MDT(\rho_\MDT)=\Mmax(\SM')=\{M_1,M_2,\ldots,M_\ell\}$ for $\ell\geq2.$ If $\ell\geq3$,
	then by Lemma~\ref{lem:sets12same}, $\SM'$ is not a $(v,\SM_1,\SM_2)$-polar-cat and
	\texttt{Check\_polar-cat} thus correctly continues to the next vertex of $H$. 
	
	If, instead, a	trivial look-up in $\MDT$ reveals that $\ell=2$, we define $\SM_1$ and $\SM_2$ to be the induced
	substrudigrams $\SM[M_1\cup\{v\}]$ respectively $\SM[M_2\cup\{v\}]$ in Line~\ref{alg:line: def substrud};
	a task that can be done again in polynomial time.
	By	Theorem~\ref{thm:kseries=k'series} we have $\SM'=(\SM_1-v)\join_k(\SM_2-v)$, where $k=t(\rho_T)$,
	whenever \texttt{Check\_polar-cat} reaches Line~\ref{alg:line: def substrud}. At this point,
	$\SM'$ is thus a $(v,\SM_1,\SM_2)$-polar-cat if and only if Definition~\ref{def:polarcat}(ii) is
	satisfied; a task that is accomplished in the \emph{if}-clause in Line~\ref{alg:line: check pc}. 
	In particular,	 we must check if $\SM_i$
	is explained by a labeled discriminating caterpillar tree $(T_i ,t_i)$, for $1\leq i \leq 2$.
	As shown in \cite{Hellmuth:13a}, checking if there exists a labeled discriminating tree $(T_i
	,t_i)$ that explains $\SM_i$, and computing it in the affirmative case, can be done in
	polynomial time. In particular, these trees $(T_1 ,t_1)$ and 
	$(T_2	,t_2)$ are unique (up to isomorphism), see \cite[Thm.~2]{Hellmuth:13a} or \cite{BD98}.
	It is now an easy task to verify if $(T_i ,t_i)$ is a caterpillar tree in
	polynomial-time. To verify Def.\ \ref{def:polarcat}(ii) we finally check that $t(\rho_{T_i})\neq
	k$ and that $v$ is part of a cherry in $T_i$ for each $i\in \{1,2\}$ by a simple polynomial-time
	traversal of $T_i$. In summary, the condition in Line~\ref{alg:line: check pc} can be checked in
	polynomial time, and $\SM$ is, at this point, a $(v,\SM_1,\SM_2)$-polar-cat if and only if this
	condition evaluates to \texttt{True}. In particular, \texttt{Check\_polar-cat} continues to the
	next vertex of $H$ if $\SM$ is not a $(v,\SM_1,\SM_2)$-polar-cat and will, in extension, always
	return \texttt{False} in Line~\ref{alg:line:return-false} whenever there is no vertex $v\in
	V(H)$, i.e. no vertex $v\in X$, such that $\SM$ is a $(v,\SM_1,\SM_2)$-polar-cat. Whenever the
	condition on Line~\ref{alg:line: check pc} holds, it is easy to construct the galled-tree
	$(N,t)\coloneqq\N(v,\SM_1,\SM_2)$ from the caterpillar trees $(T_1,t_1)$ and $(T_2,t_2)$ in polynomial-time, which,
	by Theorem~\ref{lem:pc=>gatex}, explains $\SM$. Hence $(N,t)$ is correctly returned in
	Line~\ref{alg:line:ret galledtree}. As argued, every step of the \emph{for}-loop of
	\texttt{Check\_polar-cat} takes polynomial-time, concluding an overall polynomial runtime of
	\texttt{Check\_polar-cat}.
\end{proof}

Our construction of galled-trees that are not elementary relies on pvr-networks. To prove that the
construction of a pvr-network from a prime-explaining family can be implemented in polynomial time,
we need to make some general assumptions about how networks are represented and stored in memory;
for simplicity, we assume that each network $(N,t)$ is represented such that each vertex $v\in V(N)$
has access to arrays $v.\mathtt{child}$ resp. $v.\mathtt{parent}$ of pointers to its children
respectively parents and, moreover, to the value $t(v)$ stored at $v.\mathtt{t}$. Although
technically not required by Definition~\ref{def:pvr}, it is also natural to assume that the number
of vertices in a network $(N_M,t_M)$ of a prime-explaining family $\pfam$ is polynomially bounded in
terms of the number of leaves of $N_M$; to be more precise, we assume that $|V(N_M)|\in
O(p(|L(N_M)|))$ for some polynomial $p$. This is motivated by Proposition~\ref{prop:SM-hasExpl}
which, to recall, ensures that there exists a network with $O(|L(N_M)|^2)$ vertices explaining the
same strudigram as $(N_M,t_M)$. In particular, elementary galled-trees with $\ell$ leaves have
$O(\ell)$ vertices and edges.

\begin{proposition}\label{prop:pvr-in-pol-time}
	Given a prime-explaining family $\pfam$ of a strudigram $\SM$, construction of the pvr-network $\pvr(\SM,\pfam)$ can be implemented to run in polynomial time.
\end{proposition}
\begin{proof}
	Given a strudigram $\SM=(X,\Upsilon,\sigma)$ with $|X|=n$, we can construct its MDT $(\MDT,\tau)$
	in polynomial time \cite{EHREN1994}. Suppose, additionally, that a prime-explaining family
	$\pfam=\{(N_M,t_M)\mid M\in\P\}$ of $\SM$ is given, where $\P$ is the set of prime module of
	$\SM$. Note that since the MDT of $\SM$ is phylogenetic, it contains at most $n-1$ inner-vertices
	in addition to its $n$ leaves, thus $|V(T)|\in O(n)$ and $|E(T)|\leq 2n-2\in O(n)$.
	In particular,
	$|\P|\in O(n)$. Starting a traversal of $\MDT$ at its root, at each vertex $M\in \P$ we replace
	$M.\mathtt{child}$ (but temporarily storing its content in an array $A$) with the elements of
	$\rho_{N_M}.\mathtt{child}$ and put $v.\mathtt{t}$ to be $\rho_{N_M}.\mathtt{t}$, followed by a
	traversal of $N_M$ to locate the elements of $L(N_M)$. For each $M'\in L(N_M)$, locate the vertex
	$v$ of $\MDT$ in $A$ that is to be identified with the vertex $M'$, and replace the array
	$v.\mathtt{parent}$ with the content of $M'.\mathtt{parent}$. It is easy to see that 
	after $\MDT$ is fully traversed, the resulting network aligns with
	Definition~\ref{def:pvr}. Since each step of the traversal of the $O(n)$ vertices in
	$\MDT$ depend only on $|\P|$, $|A|$ and the number of vertices and edges in $N_M$, all which are
	polynomially bounded, an implementation of these tasks can be done in polynomial time.
\end{proof}

Finally, we show that general $\gatex$ strudigrams can be recognized
and the construction of networks explaining $\gatex$ strudigrams can be done
in polynomial time.

\begin{proposition}
	It is possible to correctly verify if a strudigram $\SM$ is \textup{\gatex} and, in the affirmative, construct a labeled galled-tree that explains $\SM$. Both tasks can be implemented to run in polynomial time.
\end{proposition}
\begin{proof}
	Let $\SM$ be a strudigram. The modular decomposition $\MDstrong$ of $\SM$ can be 
	can be determined in polynomial time \cite{EHREN1994}. In particular, the set $\P\subseteq\MDstrong$ of all prime modules
	of $\SM$ can be obtained in polynomial time. Moreover, given some $M\in\P$ the computation of a
	strudigram that is cr-isomorphic to the quotient $\SM_M\coloneqq \SM[M]/\Mmax(\SM[M])$ can, by
	Observation~\ref{obs:quotient}, be done by constructing a set $W$ that contains precisely one
	element from each module in $\Mmax(\SM[M])\subseteq\MDstrong$ and then considering the strudigram
	$\SM[W]$. Thus construction of $\SM_M$ is possible in polynomial time. Recall that by
	Theorem~\ref{thm:Mprime=>QuotientMtruly-primitive} the quotient strudigram $\SM_M$ is
	truly-primitive. This together with  Lemma~\ref{lem: pol-check-pc} implies that one can correctly verify whether $\SM_M$ is a
	handsome polar-cat in polynomial time. Since $M$ was arbitrarily chosen, and $\SM$ is, by
	Theorem~\ref{thm:CharprimeCat}, \gatex if and only if $\SM\in\PrimeCat$, it is thus possible to
	decide weather or not $\SM$ is \gatex in polynomial time.
	
	Suppose now that $\SM$ is \gatex, that is, when $\SM\in\PrimeCat$. Using the notation from the
	previous paragraph, we note that by Lemma~\ref{lem: pol-check-pc}, the algorithm
	\texttt{Check\_polar-cat} will output a labeled galled-tree $(N_M,t_M)$ that explains $\SM_M$ for
	each $M\in\P$. Hence, if $\SM$ is \gatex, it is possible to find a prime-explaining family
	$\pfam=\{(N_p,t_p)\}_{p\in\P}$ of $\SM$ containing only galled-trees. By
	Proposition~\ref{prop:pvr-in-pol-time}, it is possible to construct the pvr network
	$(N,t)\coloneqq\pvr(\SM,\pfam)$ in polynomial time. Finally, by
	Proposition~\ref{prop:pvr-explains-S} respectively Lemma~\ref{lem:galled-tree-pvr}, $(N,t)$ is a
	galled-tree that explains $\SM$.
\end{proof}

\section{Summary and Outlook}

In this work, we considered strudigrams, which encompass several combinatorial objects such as
symbolic date maps, labeled symmetric 2-structures, and edge-colored undirected graphs.
Specifically, we provided a polynomial-time approach to construct labeled networks $(N,t)$ that
explain an arbitrary given strudigram $\SM$. Furthermore, we demonstrated that every strudigram can
be explained by an lca-network. These networks $N$ have the advantageous property that the least
common ancestor (lca) for all $A \subseteq L(N)$ is well-defined. However, such networks can be
quite complex, as illustrated in Fig.~\ref{fig:lca-N}. This raises the question of whether simpler
networks can explain $\SM$. To address this, we considered galled-trees. We characterized
strudigrams that can be explained by galled-trees (\gatex strudigrams) and provided polynomial-time
algorithms for their recognition and the construction of labeled galled-trees explaining them. For
this purpose, we used prime-vertex replacement (pvr) networks obtained from the modular
decomposition tree by replacing prime vertices with labeled galls.

This paper opens many avenues for future research. One promising direction is to establish a
characterization of \gatex strudigrams in terms of a family of forbidden substrudigrams. Since the
property of being \gatex is hereditary, such a family exists but could be of infinite size.
Nevertheless, for \gatex strudigrams $\SM= (X,\Upsilon,\sigma)$ with $|\Upsilon|=2$, a
characterization was provided in \cite[Theorem 4.8]{HS:24} by means of a set of 25 forbidden
subgraphs, the largest of which has a size of 8. It would be interesting to investigate whether this
bound still holds for larger sets $\Upsilon$, and if not, to determine how this bound relates to the
size of $\Upsilon$.

In Sections~\ref{sec:gatex} and \ref{sec:algo}, we focused on galled-trees, a subclass of level-1
networks, which are networks where each biconnected component contains at most one hybrid. It would
be interesting to characterize strudigrams that can be explained by level-1 or more general level-k
networks. In this context, we believe that pvr networks become crucial, and one of the main tasks is
to characterize the structure of primitive strudigrams that can be explained by such networks.

Generally, one might be interested in networks $N$ with the $\2lca$-property that are free from
``superfluous'' vertices that do not serve as the $\lca$ for any $x, y \in L(N)$. For example, in
Fig.~\ref{fig:lca-N}, the $\lca$ of any two vertices is located in the network $N(w)$ rooted at $w$. It
would be valuable to explore constructive methods to stepwise remove such superfluous vertices. In
particular, one may seek networks with a minimum number of vertices or hybrids that explain a given
strudigram. Are these problems NP-hard or solvable in polynomial time?

Regarding truly primitive strudigrams, Theorem~\ref{thm:unp<=>treeExpl} implies that such type of
strudigrams must contain at least one induced $P_4$ or rainbow triangle. Observations suggest that
these substructures are quite prevalent in truly primitive strudigrams, with a significant
proportion of vertices involved in at least one $P_4$ or rainbow triangle. To gain a deeper
understanding of the structure of truly primitive strudigrams, it would be of interest to study the
density and distribution of these substructures in greater detail. 
Additionally, it is worth investigating to what extent truly primitive strudigrams $\SM =
(X,\Upsilon,\sigma)$ are explained by unique galled-trees, a question that has already been answered
for the case $|\Upsilon|=2$ \cite[Prop.\ 6.13]{HS:22}.

\paragraph{Acknowledgements}
This paper is dedicated to the memory of Andreas Dress, whose pioneering work and dedication to discrete biomathematics, as well as many other fields, have inspired generations of researchers. His contributions will undoubtedly continue to influence and guide our work for many years to come.

\section*{Declarations}

The authors declare that they have no competing interests. All authors contributed equally to this work.

\bibliographystyle{spbasic}
\bibliography{pc}

\end{document}